\documentclass[a4paper]{amsart}

\usepackage[T1]{fontenc}
\usepackage[utf8]{inputenc}
\usepackage[english]{babel}
\usepackage{csquotes}
\usepackage[
	doi=false,
	giveninits=false,
	isbn=false,
	maxalphanames=99,
	maxbibnames=99,
	style=alphabetic]
{biblatex}
\bibliography{literature.bib}

\usepackage{mathtools}
 
\usepackage{bbm}
\usepackage{mathrsfs}

\usepackage{amsmath, amssymb, amsthm}

\makeatletter
\def\blfootnote{\xdef\@thefnmark{}\@footnotetext}
\makeatother

\newcommand{\abs}[1]{\left| #1 \right|}

\newcommand{\Cb}{C_{\text{b}}}
\newcommand{\CN}{\mathbb C}
\newcommand{\D}{\mathrm{d}}
\newcommand{\dup}[2]{\left\langle #1, #2 \right\rangle}
\newcommand{\FT}{\mathcal{F}}
\newcommand{\ga}{\gamma_{\text{a}}}
\newcommand{\Ind}{\mathbbmss{1}}
\newcommand{\iu}{\mathrm{i}}
\newcommand{\iv}[1]{\frac{1}{ #1}}
\newcommand{\jb}[1]{\langle #1 \rangle}
\newcommand{\LSS}{\mathcal{S}}
\newcommand{\N}{{\mathbb N}}
\newcommand{\norm}[1]{\left\Vert #1 \right\Vert}
\newcommand{\opT}{\mathcal{T}}
\newcommand{\qa}{q_{\text{a}}}
\newcommand{\R}{{\mathbb R}}
\renewcommand{\Re}{\operatorname{Re}}
\newcommand{\set}[1]{\left\{ #1 \right\}}

\newcommand{\T}{{\mathbb T}}

\newcommand{\Z}{{\mathbb Z}}

\theoremstyle{plain}
\newtheorem{thm}{Theorem}
\newtheorem{lem}[thm]{Lemma}
\newtheorem{prop}[thm]{Proposition}
\newtheorem{rem}[thm]{Remark}

\theoremstyle{definition}

\usepackage[disable]{todonotes}

\usepackage{IEEEtrantools}

\usepackage{hyperref}

\title{On the global well-posedness of the quadratic NLS
on $L^2(\R) + H^1(\T)$}

\author{L. Chaichenets}
\address{Leonid Chaichenets, Department of Mathematics,
Institute for Analysis, Karlsruhe Institute of Technology,
76128 Karlsruhe, Germany}
\email{leonid.chaichenets@kit.edu}
\author{D. Hundertmark}
\address{Dirk Hundertmark, Department of Mathematics,
Institute for Analysis, Karlsruhe Institute of Technology,
76128 Karlsruhe, Germany}
\email{dirk.hundertmark@kit.edu}
\author{P. Kunstmann}
\address{Peer Christian Kunstmann, Department of Mathematics,
Institute for Analysis, Karlsruhe Institute of Technology,
76128 Karlsruhe, Germany}
\email{peer.kunstmann@kit.edu}
\author{N. Pattakos}
\address{Nikolaos Pattakos, Department of Mathematics,
Institute for Analysis, Karlsruhe Institute of Technology,
76128 Karlsruhe, Germany}
\email{nikolaos.pattakos@kit.edu}

\begin{document}
{\let\thefootnote\relax\footnote{
\copyright 2019 by the authors. Faithful reproduction of
this article, in its entirety, by any means is permitted for
noncommercial purposes.}}

\subjclass[2010]{35A01, 35A02, 35Q55.}
\keywords{Nonlinear Schrödinger equation, local well-posedness,
global well-posedness, Gronwall’s inequality, Strichartz estimates.}

\begin{abstract}
We study the one dimensional nonlinear Schrödinger equation with power
nonlinearity $\abs{u}^{\alpha - 1} u$ for $\alpha \in [1,5]$ and
initial data $u_0 \in L^2(\R) + H^1(\T)$. We show via Strichartz
estimates that the Cauchy problem is locally well-posed. In the case
of the quadratic nonlinearity ($\alpha = 2$) we obtain \emph{global}
well-posedness in the space $C(\R, L^2(\R) + H^1(\T))$ via Gronwall’s
inequality.
\end{abstract}

\maketitle

\section{Introduction and main results}
We are interested in the Cauchy problem for the nonlinear Schrödinger
equation (NLS) with power nonlinearity on the space $L^2(\R) + H^1(\T)$,
i.e. \begin{equation}
\label{eqn:cauchy_nls}
\left\{
\begin{IEEEeqnarraybox}[][c]{rCl}
\iu u_t (x, t) + \partial_{x}^{2} u (x,t) \pm
\abs{u}^{\alpha-1} u & = & 0 \qquad (x, t) \in \R \times \R, \\
u(\cdot, 0) & = & u_0,
\end{IEEEeqnarraybox}
\right.
\end{equation}
where $u_0 = v_0 + w_0 \in L^2(\R) + H^1(\T)$ and $\alpha \in [1,5]$. By
$\T$ we denote the one-dimensional torus, i.e. $\T = \R / 2\pi\Z$, where 
we consider functions on $\T$ to be $2\pi$-periodic functions on $\R$.
Before we state our main results, let us mention that the NLS
\eqref{eqn:cauchy_nls} is globally well-posed in $L^2(\R)$ via
Strichartz estimates and mass conservation (see \cite{tsutsumi1987})
and it is globally well-posed in $L^{2}(\T)$ via the Fourier restriction
norm method and mass conservation (see \cite{bourgain1993a}). 
Motivation for the investigation of hybrid initial values
$u_0 \in L^2(\R) + H^1(\T)$ comes from high--speed optical fiber 
communications, where in a certain approximation the behavior of 
pulses in glass--fiber cables is described by a NLS equation.  
The NLS \eqref{eqn:cauchy_nls} with initial data in $H^s(\R) + H^s(\T)$ was 
referred to in \cite{pattakos2018} as the tooth problem. A tooth is, for example, $w_0$ 
restricted to one period. We think of the addition of $v_0$ to $w_0$ as eliminating
finitely many of these teeth in the underlying periodic signal. 
A periodic signal is the simplest type of a non-decaying signal,
encoding, for example, an infinite string of ones 
if there is exactly one tooth per period. However, such 
a purely periodic signal carries no information. One would like to be
able to change it, at least locally. This leads necessarily to a hybrid
formulation of the NLS where the  signal is the sum of a periodic and a
localized part, the localized part being able to remove one or more of
the teeth in the underlying periodic signal. This way one can model, 
for example, a signal consisting of two infinite blocks of ones which
are separated by a single zero, or even far more complicated patterns.
In the optics literature the phenomenon of ghost pulses (see
\cite{mamyshev1999} and \cite{zakharov1999}) occurs which in our
terminology corresponds to the regrowth of missing teeth of the solution
to the NLS \eqref{eqn:cauchy_nls}.

The case of the cubic nonlinearity ($\alpha=3$) and the initial data
$u_0 \in H^s(\R) + H^s(\T)$, where $s \geq 0$, was studied by the
authors in \cite{pattakos2018}, where the existence of weak solutions
in the extended sense was established. Moreover, under some further
assumptions, unconditional uniqueness was obtained. In this paper, due
to the non-algebraic structure of the nonlinearity in
\eqref{eqn:cauchy_nls} (for $\alpha \neq 3,5$) we have to use different
methods. For the relation between the solutions of \cite{pattakos2018}
and the solutions of Theorem \ref{thm:mainthm1} we refer to Remark
\ref{equal_solutions}. 

To state the main results of this paper we need some preparation. 
Let $u = v + w \in C([0, T], L^2(\R) + H^1(\T))$ where $w$ satisfies the
periodic NLS on the torus with initial data $w_0$. The following is
known about $w$ (the case $\alpha \geq 2$ has been treated in
\cite[Theorem 2.1]{lebowitz1988} while the remaining case
$\alpha \in [1, 2)$ is presented in Theorem
\ref{thm:gwp_cauchy_quadratic_nls_torus}).
\begin{thm}
\label{thm:lwp_nls_T}
The Cauchy problem for the periodic NLS
\begin{equation}
\label{eqn:cauchy_nls_torus}
\left\{
\begin{IEEEeqnarraybox}[][c]{rCl}
\iu w_t (x, t) + \partial_{x}^{2} w (x,t) \pm
\abs{w}^{\alpha-1} w & = & 0 \qquad (x, t) \in \T \times \R, \\
w(\cdot, 0) & = & w_0.
\end{IEEEeqnarraybox}
\right.
\end{equation}
is locally well-posed in $H^1(\T)$ for $\alpha \geq 1$. That means that
for any $w_0 \in H^1(\T)$ there is a unique $w \in C([0,T], H^1(\T))$
satisfying \eqref{eqn:cauchy_nls_torus} in the mild sense. The
guaranteed time of existence $T$ depends only on $\norm{w_0}_{H^1(\T)}$.
\end{thm}

A solution $w$ to the periodic NLS at hand dictates that the local part $v$ has to be a solution of the Cauchy problem for
the \emph{modified NLS}
\begin{equation}
\label{eqn:cauchy_modnls}
\left\{
\begin{IEEEeqnarraybox}[][c]{rCl}
\iu v_t (x, t) + \partial_{x}^{2} v (x,t) \pm
G_\alpha (w, v) & = & 0 \qquad (x, t) \in \R \times \R, \\
v(\cdot, 0) & = & v_0,
\end{IEEEeqnarraybox}
\right.
\end{equation}
where
\begin{equation}
G_\alpha (w, v) \coloneqq 
\abs{v + w}^{\alpha-1}(v + w) - \abs{w}^{\alpha - 1} w.
\end{equation}
The main results of the paper are the following two theorems on local and global wellposedness of NLS \eqref{eqn:cauchy_modnls} and consequently NLS \eqref{eqn:cauchy_nls}.

\begin{thm}[Local well-posedness of the NLS \eqref{eqn:cauchy_nls}]
\label{thm:mainthm1}
For $\alpha \in [1, 5]$ the Cauchy problem \eqref{eqn:cauchy_modnls}
is locally well-posed in $C([0,T], L^2(\R)) \cap
L^{\frac{4(\alpha+1)}{\alpha-1}}([0,T], L^{\alpha + 1}(\R))$ for any
$v_0 \in L^{2}(\R)$.

Hence, the original Cauchy problem \eqref{eqn:cauchy_nls} is locally
well-posed.

In the case $\alpha \in [1, 5)$, the guaranteed time of existence $T$
depends only on $\norm{v_0}_2$ and $\norm{w_0}_{H^1(\T)}$, whereas, for
$\alpha = 5$, $T$ depends on the profile of $v_0$ and
$\norm{w_0}_{H^1(\T)}$.
\end{thm}

\begin{rem}
In the case $\alpha \in [1, 2]$, the intersection in Theorem
\ref{thm:mainthm1} is not needed, i.e. one has unconditional
well-posedness for the perturbation $v$. However, it is not clear whether
the Cauchy problem \eqref{eqn:cauchy_nls} is unconditionally
well-posed, since the wellposedness we obtain for the
periodic part $w$ is only conditional (see the proof of Theorem \ref{thm:lwp_torus}).
\end{rem}

\begin{rem}
\label{equal_solutions}
Notice that the weak solution in the extended sense $\tilde{u}$
constructed in \cite{pattakos2018} and the solution $u$ from Theorem
\ref{thm:mainthm1} coincide. This can be seen as follows: $u$ is a weak
solution in the extended sense, which follows by the definition,
Plancherel’s theorem and the dominated convergence theorem. Moreover, in
the aforementioned paper it was observed that $\tilde{u}$ is unique
among those solutions, which can be approximated by smooth solutions.
This is true for $u$ and hence $\tilde{u} = u$ follows.
\end{rem}

For $\alpha = 2$, we need the Cauchy problem for the periodic NLS
\eqref{eqn:cauchy_nls_torus} to be globally well-posed in $H^{1}(\T)$.
Although this is claimed to be well-known in the community, we could
not find a suitable reference. Several people refer to
\cite{bourgain1993a} for this, however in \cite[Proposition
  5.73]{bourgain1993a} $\alpha\geq3$ is required (in our
notation). Moreover, in part ii) of the remark on page 152 in
\cite{bourgain1993a}, Bourgain mentions that one could get existence
of a solution for the quadratic nonlinearity using Schauder’s fixed
point theorem, but one would loose uniqueness.  Hence, we provide a
proof in the Appendix (Theorem
\ref{thm:gwp_cauchy_quadratic_nls_torus}). This global existence and
uniqueness result on the torus, together with a close inspection of
the mass $\int \abs{v}^2 \D{x}$ are essential ingredients in our proof
of global well-posedness of \eqref{eqn:cauchy_nls} on the “tooth problem
space” $L^2(\R) + H^1(\T)$.

\begin{thm}[Global well-posedness of the quadratic NLS]
\label{thm:mainthm2}
For $\alpha = 2$ and $v_{0}\in L^{2}(\R)$ the unique solution $v$ of
\eqref{eqn:cauchy_modnls} from Theorem \ref{thm:mainthm1} extends
globally and obeys the bound
\begin{equation}
\label{eqn:exp_bound}
\norm{v(\cdot, t)}_2 \leq
\norm{v_0}_2 \exp \left[\norm{w}_{L^\infty_t L^\infty_x} t \right]
\qquad \forall t \in [0, \infty).
\end{equation}
Hence, the original Cauchy problem \eqref{eqn:cauchy_nls} for
$\alpha = 2$ is globally well-posed.
\end{thm}

Although the local well-posedness result of Theorem \ref{thm:mainthm1}
covers the whole range $\alpha \in [1, 5]$, the methods of the proof of
Theorem \ref{thm:mainthm2} only work for $\alpha = 2$. A more precise
explanation is given in Remark \ref{rem:failure_global}.

Of course, one can consider hybrid problems for other dispersive equations. 
Here we confine ourselves to a remark on the KdV. 

\begin{rem}
Observe that the tooth problem for the KdV reduces to a known setting.
More precisely, consider real solutions of
\begin{equation}
\label{tooth_KdV}
\left\{
\begin{IEEEeqnarraybox}[][c]{rCl}
u_t (x, t) + u_{xxx} (x,t) + u_x u & = & 0
\qquad (x, t) \in \R \times \R, \\
u(\cdot, 0) & = & u_0 = v_0 + w_0 \in H^{s_1}(\R) + H^{s_2}(\T).
\end{IEEEeqnarraybox}
\right.
\end{equation}
Let $u = v + w \in C([0, T], H^{s_1}(\R) + H^{s_2}(\T))$,
where $s_2 \in \N$ and $w$ is a global solution of the periodic KdV for
the initial data $w_0$ (see \cite[Theorem 5]{bourgain1993b}). Then $v$
solves
\begin{equation*}
v_t + v_{xxx} + v_x v + (w v)_x = 0
\end{equation*}
with the initial data $v_0$, which is the KdV with the potential $w$.
This problem has been studied in e.g. \cite[Section 3.1]{erdogan2016}
using parabolic regularization. There it has been shown that
$v$ satisfies an exponential bound similar to \eqref{eqn:exp_bound}.
Combining both results we obtain:
\begin{quote}
For $s_1, s_2 \in \N$ satisfying $s_1 \geq 2$ and
$s_2 \geq s_1 + 1$ the KdV tooth problem, i.e., the Cauchy problem \eqref{tooth_KdV}, is globally
well-posed in $H^{s_1}(\R) + H^{s_2}(\T)$.
\end{quote}
\end{rem}

The paper is organized as follows: In Section \ref{section_pre} 
we state the required prerequisites for the proofs of the main theorems. 
In Section \ref{sec:lwp} we present
the proof of Theorem \ref{thm:mainthm1} and in Section
\ref{sec:quadraticglobal} we present the proof of Theorem
\ref{thm:mainthm2}. Finally, in the Appendix we justify that the quadratic and
subquadratic periodic NLS \eqref{eqn:cauchy_nls_torus} is globally
well-posed in $H^{1}(\T)$. 

\section{Prerequisites}
\label{section_pre}
Let us fix the notation and state some results necessary for the proof
of our main theorems. For the purpose of smoothing we will use the heat
kernel $(\phi_\varepsilon)_{\varepsilon \geq 0}$. Recall, that
$\phi_\varepsilon = \delta_0$, if $\varepsilon = 0$, and
\begin{equation*}
\phi_\varepsilon(x) = \iv{2 \sqrt{\pi \varepsilon}} \,
e^{-\frac{\abs{x}^2}{4 \varepsilon}} \qquad
\forall x \in \R,
\end{equation*}
if $\varepsilon > 0$. We shall denote the convolution (in the space
variable $x$) by e.g. $u \ast \phi_\varepsilon$. 

For $s \in \R$ and $\Omega \in \set{\R, \T}$ we shall denote by
$H^s(\Omega)$ the Sobolev spaces on $\Omega$. Also, we set
$H^\infty(\Omega) \coloneqq \cap_{s \in \R} H^s(\Omega)$. By $\FT$ we
will denote the Fourier transform on $\R$.

We will use the following simple lemma, which can be found e.g. in
\cite[Lemma 3.9]{chaichenets2018}.
\begin{lem}[Size estimate]
\label{lem:size_estimate}
Let $\alpha \geq 1$. Then the following \emph{size estimate}
\begin{eqnarray}
\label{eqn:size_estimate}
& & \abs{\abs{v_1 + w}^{(\alpha - 1)}(v_1 + w) -
\abs{v_2 + w}^{(\alpha - 1)}(v_2 + w)} \\
& \leq &
\alpha \max \set{1, 2^{\alpha - 1}} 
\left(\abs{v_1}^{\alpha - 1} + \abs{v_2}^{\alpha - 1} +
\abs{w}^{\alpha - 1} \right) \abs{v_1 - v_2}
\end{eqnarray}
holds for any $v_1, v_2, w \in \CN$.
\end{lem}

A pair of exponents $(r, q) \in [2, \infty]^2$ is called
\emph{admissible} (in one dimension), if
\begin{equation}
\label{eqn:admissible}
\frac{2}{q} + \iv{r} = \frac{1}{2}.
\end{equation}
Let us denote by $\qa(r)$ the solution of \eqref{eqn:admissible} for any
$r \in [2, \infty]$. Another pair of exponents
$(\rho, \gamma) \in [1, 2]$ shall be called \emph{dually admissible}, if
$(\rho', \gamma') \in [2, \infty]$ is admissible, i.e. if
\begin{equation}
\label{eqn:dually_admissible}
\frac{2}{\gamma} + \iv{\rho} = \frac{5}{2}.
\end{equation}
For any $\rho \in [1, 2]$ we denote by $\ga(\rho)$ the solution of \eqref{eqn:dually_admissible}.

\begin{prop}[Strichartz estimates]
(Cf. \cite[Theorem 1.2]{keel1998})
\label{prop:strichartz}
Let $(r, \qa(r))$ be admissible and $(\rho, \ga(\rho))$ be dually
admissible. Then there is a constant $C = C(r, \rho) > 0$ such that for
any $T > 0$, any $v_0 \in L^2(\R)$ and any
$F \in L^\ga(\rho)([0, T], L^\rho(\R))$ the \emph{homogeneous} and
\emph{inhomogeneous} Strichartz estimates
\begin{eqnarray}
\label{eqn:homogeneous_strichartz}
\norm{e^{\iu t \partial_{x}^{2}} v_0}_{L^{\qa(r)}([0,T], L^r(\R))} & \leq &
C \norm{v_0}_{L^2(\R)}, \\
\label{eqn:inhomogeneous_strichartz}
\norm{\int_0^t
e^{\iu (t - \tau) \partial_{x}^{2}} F(\cdot, \tau)}_{L^{\qa(r)}([0,T], L^r(\R))}
& \leq &
C \norm{F}_{L^{\ga(\rho)}([0,T], L^\rho(\R))}.
\end{eqnarray}
hold.
\end{prop}

\begin{lem}[Gronwall, integral form]
\label{lem:gronwall_integral}
(See \cite[Theorem 1.10]{tao2006}.)
Let $A, T \geq 0$ and $u, B \in C([0, T], \R^+_0)$ be such that
\begin{equation*}
u(t) \leq A + \int_0^t B(s) u(s) \D{s} \qquad \forall t \in [0, T].
\end{equation*}
Then
\begin{equation*}
u(t) \leq A \exp\left(\int_0^t B(s) \D{s}\right) \qquad
\forall t \in [0, T].
\end{equation*}
\end{lem}

\begin{lem}[Gronwall, differential form]
\label{lem:gronwall_diff}
(See \cite[Theorem 1.12]{tao2006}.)
Let $T > 0$, $u: [0, T] \to \R^+_0$ be absolutely continuous
and $B \in C([0,T], \R_0^+)$ such that
\begin{equation*}
u'(t) \leq B(t) u(t) \qquad
\text{for almost every } t \in [0, T].
\end{equation*}
Then
\begin{equation*}
u(t) \leq u(0) \exp\left(\int_0^t B(s) \D{s}\right) \qquad
\forall t \in [0, T].
\end{equation*}
\end{lem}

\begin{lem}
\label{lem:fac}
(See \cite[Equation (18)]{pattakos2018}).
Let $s \geq 0$. Then there is a constant $C = C(s) > 0$ such that for
any $v \in H^s(\R)$ and any $w \in H^{s + 1}(\T)$ one has
$v \cdot w \in H^s(\R)$ and
\begin{equation*}
\norm{vw}_{H^s(\R)} \leq C \norm{v}_{H^s(\R)} \norm{w}_{H^{s + 1}(\T)}.
\end{equation*}
\end{lem}
The above estimate is not optimal w.r.t. the assumed regularity of $w$.
However, we do not need a stronger version and the proof is straight
forward.

\section{Proof of Theorem \ref{thm:mainthm1}}
\label{sec:lwp}
Consider first the case $\alpha \in [2, 5)$. Let us define the space
\begin{equation}
\label{eqn:contraction_space}
X \coloneqq C([0, T], L^2(\R)) \cap
L^{\qa(\alpha + 1)}([0,T], L^{\alpha + 1}(\R))
\end{equation}
equipped with the norm
\begin{equation*}
\norm{v}_X \coloneqq
\norm{v}_{L^\infty_t L^2_x} + \norm{v}_{L^{\qa(\alpha + 1)}_t L^{\alpha + 1}}
\qquad \forall v \in X,
\end{equation*}
where $T$ will be fixed later in the proof. The integral formulation of
\eqref{eqn:cauchy_modnls} reads as
\begin{equation}
\label{eqn:duhamel_modnls}
v =
e^{\iu t \partial_{x}^{2}} v_0 \pm
\iu \int_0^t e^{\iu (t - \tau) \partial_{x}^{2}}
G_\alpha(w, v) \D{\tau}
\eqqcolon \opT(v).
\end{equation}

By Banach’s fixed-point theorem, it suffices to show that there are
$R, T > 0$ such that $\opT$ is a contractive self-mapping of
\begin{equation*}
M(R, T) \coloneqq \set{v \in X \Big| \, \norm{v}_X \leq R}.
\end{equation*}
Consider first the self-mapping property. For
$r \in \set{2, \alpha + 1}$ we have
\begin{equation*}
\norm{\opT v}_{L^{\qa(r)}_t L^r_x} \leq
\norm{e^{\iu t \partial_{x}^{2}} v_0}_{L^{\qa(r)}_t L^r_x} +
\norm{\int_0^t e^{\iu (t - \tau) \partial_{x}^{2}}
G_\alpha(w, v) \D{\tau}}_{L^{\qa(r)}_t L^r_x}.
\end{equation*}
By the homogeneous Strichartz estimate
\eqref{eqn:homogeneous_strichartz}, we have
\begin{equation*}
\norm{e^{\iu t \partial_{x}^{2}} v_0}_{L^{\qa(r)}_t L^r_x} \lesssim \norm{v_0}_2
\end{equation*}
for the first summand. This suggests the choice
$R \approx \norm{v_0}_2$. For the second summand, whose norm also needs
to be comparable with $R$, we will split the integral term. We
proceed with the estimates for the contraction property of $\opT$,
because the self-mapping property follows from them by setting $v = v_1$
and $v_2 = 0$. To that end, let us define
$G_\alpha(w, v_1, v_2) \coloneqq G_\alpha(w, v_1) - G_\alpha(w, v_2)$,
set
\begin{equation}
\label{eqn:char_set}
A \coloneqq
\set{x \in \R \, | \, \abs{w} \leq
(\abs{v_1} + \abs{v_2})},
\end{equation}
and introduce
\begin{equation*}
G_{\alpha, 1}(w, v_1, v_2)
\coloneqq
\Ind_{A}
\left(\abs{w + v_1}^{\alpha - 1}
(w + v_1) -
\abs{w + v_2}^{\alpha - 1}
(w + v_2)
\right)
\end{equation*}
and
\begin{equation*}
G_{\alpha, 2}(w, v_1, v_2)
\coloneqq
\Ind_{A^c}
\left(\abs{w + v_1}^{\alpha - 1}
(w + v_1) -
\abs{w + v_2}^{\alpha - 1}
(w + v_2)
\right).
\end{equation*}
By the triangle inequality one obtains for $r \in \set{2, \alpha + 1}$
the estimate
\begin{eqnarray}
\nonumber
& &
\norm{
\int_0^t e^{\iu (t - \tau) \partial_{x}^{2}}
G_\alpha(w, v_1, v_2) \D{\tau}
}_{L^{\qa(r)}_t L^r_x} \\
\label{eqn:split}
& \leq &
\norm{
\int_0^t e^{\iu (t - \tau) \partial_{x}^{2}}
G_{\alpha, 1}(w, v_1, v_2)
\D{\tau}}_{L^{\qa(r)}_t L^r_x} \\
\nonumber
& & +
\norm{
\int_0^t e^{\iu (t - \tau) \partial_{x}^{2}}
G_{\alpha, 2}(w, v_1, v_2) \D{\tau}}_{L^{\qa(r)}_t L^r_x}.
\end{eqnarray}
We use the inhomogeneous Strichartz inequality and the size estimate
\eqref{eqn:size_estimate} to bound the first summand of
\eqref{eqn:split} by
\begin{eqnarray}
\label{eqn:strichartz}
& & \norm{
\int_0^t e^{\iu(t - \tau) \partial_x^2}
G_{\alpha, 1}(w, v_1, v_2)
\D{\tau}}_{L^{\qa(r)}_t L^r_x} \\
\nonumber
& \leq &
\norm{G_{\alpha, 1}(w, v_1, v_2)
}_{L^{\qa((\alpha + 1)')}_t L^{(\alpha + 1)'}_x} \\
\nonumber
& \lesssim &
\norm{
\Ind_A
\left(\abs{v_1}^{\alpha - 1} + \abs{v_2}^{\alpha - 1} +
\abs{w}^{\alpha - 1} \right)
\abs{v_1 - v_2}
}_{L^{\qa((\alpha + 1)')}_t L^{(\alpha + 1)'}_x}.
\end{eqnarray}
Using the definition of the set $A$ and Hölder’s inequality for the
space and time norms we arrive at the upper bound
\begin{equation*}
\norm{
(\abs{v_1}^{\alpha - 1} + \abs{v_2}^{\alpha - 1})
\abs{v_1 - v_2}
}_{L^{\qa((\alpha + 1)')}L^{(\alpha + 1)'}} \lesssim
T^{1 - \frac{\alpha - 1}{4}} R^{\alpha - 1} \norm{v_1 - v_2}_X.
\end{equation*}
For the second summand of \eqref{eqn:split} we obtain by the same
methods the bound
\begin{eqnarray*}
& & \norm{
\int_0^t e^{\iu (t - \tau) \partial_x^2}
G_{\alpha, 2}(w, v_1, v_2)
\D{\tau}}_{L^{\qa(r)}_t L^r_x} \\
& \lesssim &
\norm{\abs{w}^{\alpha - 1} \abs{v_1 - v_2}}_{L^1(L^2)} \lesssim
T \norm{v_1 - v_2}_X.
\end{eqnarray*}
Choosing $T$ small enough shows the contraction property of $\opT$ and
the proof, in the case $\alpha \in (2, 5)$, concludes.

The case $\alpha \in [1, 2]$ is treated in the same way, but instead of
setting $\rho = (\alpha + 1)'$ one chooses $\rho = \frac{2}{\alpha}$ for
the Strichartz exponent in \eqref{eqn:strichartz}. Applying Hölder’s
inequality subsequently leads to the $L^\infty_t L^2_x$-norm and hence
no intersection in \eqref{eqn:contraction_space} is required, i.e. we
indeed have unconditional uniqueness.

For the remaining critical case $\alpha = 5$, consider the complete
metric space
\begin{equation*}
M(R, T) \coloneqq
\set{v \in X \, \Big| \,
\norm{v - e^{\iu t \partial_x^2}v_0}_{L^\infty_t L^2_x} +
\norm{v}_{L^6_t L^6_x} \leq R}.
\end{equation*}
We have to show again that $\opT$ is a contractive self-mapping of
$M(R, T)$ for some $R, T > 0$. Candidates for $R$ and $T$ are determined
from the first term of \eqref{eqn:split}, corresponding to the
effective power $\abs{v}^5$, exactly as in the treatment of the usual
mass critical NLS (see e.g. \cite[Theorem 5.3]{linares2015}).
Subsequently,  the remaining terms corresponding to the effective power
$\abs{v}^1$ are treated via the Strichartz estimates as in the case
$\alpha \in (2, 5)$ enforcing a possibly smaller choice of $T$. We omit
the details.
\hfill\ensuremath{\square}

\section{Proof of Theorem \ref{thm:mainthm2}}
\label{sec:quadraticglobal}
The proof of Theorem \ref{thm:mainthm2} will be done by looking at the
mass $\iv{2} \norm{v(t)}_2^2$ of the solution. In order to make this
rigorous we have to work with solutions which are differentiable in
time. We will get time regularity from regularity in space. Hence we
replace $G_2$ in \eqref{eqn:cauchy_modnls} by its smooth version
$G^\varepsilon$. We obtain
\begin{equation}
\label{eqn:cauchy_modnls_smooth}
\left\{
\begin{IEEEeqnarraybox}[][c]{rCl}
\iu v_t (x, t) + \partial_{x}^{2} v (x,t) \pm
G^\varepsilon (w, v) & = & 0 \qquad (x, t) \in \R \times \R, \\
v(\cdot, 0) & = & v_0,
\end{IEEEeqnarraybox}
\right.
\end{equation}
where
\begin{equation}
G^\varepsilon(w, v) \coloneqq
[\abs{v + w} \ast \phi_\varepsilon ] (v + w) -
[\abs{w} \ast \phi_\varepsilon ]w.
\end{equation}

\begin{thm}[Local well-posedness of the smoothened modified NLS]
\label{thm:lwp_modnls}
Let $\varepsilon \geq 0$. Then there is a constant
$C > 0$ such that for any $v_0 \in L^2$ and any
$w \in C(\R, L_x^\infty)$ the Cauchy problem
\eqref{eqn:cauchy_modnls_smooth} has a unique solution in
$C([0,T], L^2(\R))$, provided
\begin{equation}
\label{eqn:lwp_time_smallness}
T \leq C \min \set{\norm{v_0}_2^{-\frac{4}{3}},
\norm{w}_{L_t^\infty, L_x^\infty}^{-1}}.
\end{equation}
\end{thm}
\begin{proof}
Consider the integral formulation of \eqref{eqn:cauchy_modnls_smooth},
i.e.
\begin{equation}
\label{eqn:duhamel_modnls_smooth}
v =
e^{\iu t \partial_{x}^{2}} v_0 \pm
\iu \int_0^t e^{\iu (t - \tau) \partial_{x}^{2}}
G^\varepsilon(w, v) \D{\tau}
\eqqcolon \opT^\varepsilon(v)
\end{equation}
and notice that 
\begin{equation*}
G^\varepsilon(w, v) =
\underbrace{
\left([\abs{v + w} - \abs{w}]\ast \phi_\varepsilon \right)v
}_{\eqqcolon G_1^\varepsilon(w, v)} +
\underbrace{
\left([\abs{v + w} - \abs{w}] \ast \phi_\varepsilon\right)w +
[\abs{w} \ast \phi_\varepsilon] v
}_{\eqqcolon G_2^\varepsilon(w, v)}.
\end{equation*}
By Banach’s fixed-point theorem, it suffices to show that there are
$R, T > 0$ such that $\opT^{\varepsilon}$ is a contractive self-mapping of
\begin{equation*}
M(R, T) \coloneqq
\set{v \in C([0, T], L^2(\R)) \Big| \, \norm{v} \leq R}.
\end{equation*}
Consider first the self-mapping property. We have
\begin{equation*}
\norm{\opT^\varepsilon v}_{L^\infty_t L^2_x} \leq
\norm{e^{\iu t \partial_{x}^{2}} v_0}_{L^\infty_t L^2_x} +
\norm{\int_0^t e^{\iu (t - \tau) \partial_{x}^{2}}
G^\varepsilon (w, v) \D{\tau}}_{L^\infty_t L^2_x}.
\end{equation*}
Since the operator $e^{it\partial_{x}^{2}}$ is an isometry on $L^{2}$ we have
\begin{equation*}
\norm{e^{\iu t \partial_{x}^{2}} v_0}_{L^\infty_t L^2_x} = \norm{v_0}_2
\end{equation*}
for the first summand. This suggests the choice
$R \approx \norm{v_0}_2$. For the second summand, whose norm needs to
also be comparable with $R$, we split the integral term and obtain
\begin{eqnarray*}
& & \norm{\int_0^T e^{\iu (t - \tau) \partial_{x}^{2}}
G^\varepsilon (w, v) \D{\tau}}_{L^\infty_t L^2_x} \\
& \leq &
\norm{\int_0^T e^{\iu (t - \tau) \partial_{x}^{2}}
G_1^\varepsilon (w, v) \D{\tau}}_{L^\infty_t L^2_x} +
\norm{\int_0^T e^{\iu (t - \tau) \partial_{x}^{2}}
G_2^\varepsilon (w, v) \D{\tau}}_{L^\infty_t L^2_x}.
\end{eqnarray*}
Now, both summands are treated via the inhomogeneous Strichartz estimate
as in the proof of Theorem \ref{thm:mainthm1}. More precisely, one has
\begin{eqnarray*}
\norm{\int_0^T e^{\iu (t - \tau) \partial_{x}^{2}}
G_1^\varepsilon (w, v) \D{\tau}}_{L^\infty_t L^2_x} & \lesssim &
\norm{([\abs{v + w} - \abs{w}] \ast \phi_\epsilon) v)
}_{L^\gamma(L^\rho)} \\
&\leq &
\norm{
\norm{[\abs{v + w} - \abs{w}] \ast \phi_\epsilon
}_{L_x^{2 \rho}}
\norm{v}_{L_x^{2 \rho}}}_{L^\gamma_t} \\
& \leq &
\norm{\norm{v}_{L_x^{2 \rho}}^2}_{L_t^{\gamma}} =
\norm{v}_{L^{2 \gamma} (L^{2 \rho})}^2.
\end{eqnarray*}
Above, we used the Cauchy-Schwartz inequality to arrive at the second
line and Young’s inequality (if $\varepsilon \neq 0$) and a size
estimate to pass to the last line (all in the space variable).

As we want to arrive at the norm in $C([0, T], L^2(\R))$, we put
$2 \rho = 2$, i.e. $\rho = 1$. Then, from the admissibility condition
\eqref{eqn:admissible} for $(\rho', \gamma')$, one obtains
$\gamma = \frac{4}{3}$. As $2 \gamma = \frac{8}{3} < \infty = \qa(2)$,
one can raise the time exponent to $\infty$ by Hölder’s inequality for
the time variable, i.e.
\begin{equation}
\label{eqn:G_1_estimate}
\norm{v}_{L^{2 \gamma} (L^{2 \rho})}^2 \leq
T^{\frac{3}{4}} \norm{v}_{L^\infty(L^2)}^2 \leq
T^{\frac{3}{4}} R^2 \overset{!}{\lesssim} R.
\end{equation}
This inequality holds under the condition
\begin{equation*}
T \lesssim \norm{v_0}_2^{-\frac{4}{3}},
\end{equation*}
which is satisfied by \eqref{eqn:lwp_time_smallness}.

For $G_2^\varepsilon$ we similarly obtain
\begin{eqnarray*}
& &
\norm{\int_0^T e^{\iu (t - \tau) \partial_{x}^{2}}
G_2^\varepsilon(w, v) \D{\tau}}_{L^\infty_t L^2_x} \\
& \lesssim &
\norm{([\abs{v + w} - \abs{w}] \ast \phi_\varepsilon])w
}_{L^{\tilde{\gamma}}(L^{\tilde{\rho}})} +
\norm{
[\abs{w} \ast \phi_\varepsilon] v
}_{L^{\tilde{\gamma}}(L^{\tilde{\rho}})} \\
& \leq &
\norm{w}_{L^\infty(L^\infty)}
\norm{[\abs{v + w} - \abs{w}] \ast \phi_\varepsilon]
}_{L^{\tilde{\gamma}}(L^{\tilde{\rho}})} +
\norm{[\abs{w} \ast \phi_\varepsilon]}_{L^\infty(L^\infty)}
\norm{v}_{L^{\tilde{\gamma}}(L^{\tilde{\rho}})} \\
& \lesssim &
\norm{w}_{L^\infty(L^\infty)}
\norm{v}_{L^{\tilde{\gamma}}(L^{\tilde{\rho}})},
\end{eqnarray*}
where we employed Young’s inequality and a size estimate to obtain the
last line. In contrast to the $G_1$-case, we choose $\tilde{\rho} = 2$
to arrive at the norm in $C([0, T], L^2(\R))$. Then, by the
admissibility condition \eqref{eqn:admissible},
$\tilde{\gamma} = 1 < \infty = \qa(2)$. Hence, by exploiting
again the Hölder’s inequality for the time variable, we get
\begin{eqnarray*}
\norm{w}_{L^\infty(L^\infty)}
\norm{v}_{L^{\tilde{\gamma}}(L^{\tilde{\rho}})} & = &
\norm{w}_{L^\infty(L^\infty)}
\norm{v}_{L^1(L^2)} \\
& \leq& 
\norm{w}_{L^\infty(L^\infty)}
T \norm{v}_{L^\infty(L^2)} \\
& \leq &
\norm{w}_{L^\infty(L^\infty)} R T \\
& \overset{!}{\lesssim_1} & R.
\end{eqnarray*}
From this we obtain the additional condition
\begin{equation*}
T \lesssim \norm{w}_{L^\infty(L^\infty)}^{-1},
\end{equation*}
which is also satisfied by \eqref{eqn:lwp_time_smallness}.

For the contraction property, consider the splitting
\begin{eqnarray*}
G^\varepsilon(w, v_1, v_2) & \coloneqq &
G^\varepsilon(w, v_1) - G^\varepsilon(w, v_2) \\
& = &
[\abs{v_1 + w} \ast \phi_\varepsilon](v_1 + w) -
[\abs{v_2 + w} \ast \phi_\varepsilon](v_2 + w) \\
& = &
\underbrace{
([\abs{v_1 + w} - \abs{w}] \ast \phi_\varepsilon)(v_1 - v_2) +
([\abs{v_1 + w} - \abs{v_2 + w}] \ast \phi_\varepsilon) v_2
}_{\eqqcolon G_1^\varepsilon(w, v_1, v_2)} \\
& & +
\underbrace{([\abs{v_1 + w} - \abs{v_2 + w}] \ast \phi_\varepsilon) w +
[\abs{w} \ast \phi_\varepsilon] (v_1 - v_2)
}_{\eqqcolon G_2^\varepsilon(w, v_1, v_2)}.
\end{eqnarray*}
Arguments similar to those used in the proof of the self-mapping property
shown above yield the contraction property of $\opT^\varepsilon$,
possibly requiring an even smaller implicit constant in
\eqref{eqn:lwp_time_smallness}.
\end{proof}

\begin{lem}[Convergence of the solutions for vanishing smoothing]
\label{lem:vanishing_smoothing}
Fix $v_0 \in L^2$ and $w \in C(\R, C(\T))$, and for all $\varepsilon \geq 0$ denote by
$v^\varepsilon \in C([0, T], L^2(\R))$ the
unique solution of the Cauchy problem \eqref{eqn:cauchy_modnls_smooth}
from Theorem \ref{thm:lwp_modnls}. Then,
\begin{equation*}
\norm{v^\varepsilon - v^0}_{L^\infty_t L^2_x}
\xrightarrow{\varepsilon \to 0+} 0.
\end{equation*}
\end{lem}
\begin{proof}
Recall, that by construction $v^\varepsilon$ and $v^0$ are fixed points
of $\opT^\varepsilon$ and $\opT^0$ respectively and hence
\begin{eqnarray*}
\norm{v^\varepsilon - v^0}_{L^\infty_t L^2_x} & \leq &
\norm{\int_0^t e^{\iu (t - \tau) \partial_{x}^{2}} \left(
G^\varepsilon(w, v^\varepsilon) -
G^0(w, v^0) \right) \D{\tau}}_{L^\infty_t L^2_x} \\
& \leq &
\norm{\int_0^t e^{\iu (t - \tau) \partial_{x}^{2}} \left(
G^\varepsilon(w, v^\varepsilon) -
G^\varepsilon(w, v^0) \right) \D{\tau}}_{L^\infty_t L^2_x} \\
& & + \norm{\int_0^t e^{\iu (t - \tau) \partial_{x}^{2}}
\left( G^\varepsilon(w, v^0) -
G^0(w, v^0) \right) \D{\tau}}_{L^\infty_t L^2_x}.
\end{eqnarray*}
Due to the fact that $\opT^\varepsilon$ is contractive, the first
summand is controlled by
\begin{equation*}
\norm{\int_0^t e^{\iu (t - \tau) \partial_{x}^{2}} \left(
G^\varepsilon(w, v^\varepsilon) -
G^\varepsilon(w, v^0) \right) \D{\tau}}_{L^\infty_t L^2_x} \leq
C \norm{v^0 - v^\varepsilon}_{L^\infty_t L^2_x},
\end{equation*}
where $C < 1$ is the contraction constant. Thus, it suffices to show
that the second summand converges to zero. To that end we first gather
terms with the same effective powers of $v^0$ and $w$, i.e.
\begin{eqnarray}
\nonumber
& & \int_0^t e^{\iu (t - \tau) \partial_{x}^{2}} \left(
G^\varepsilon(w, v^0) -
G^0(w, v^0) \right) \D{\tau}\\
\nonumber
& = &
\int_0^t e^{\iu (t - \tau) \partial_{x}^{2}} \left(
\left[\abs{w + v^0} \ast \phi_\varepsilon \right](v^0 + w) -
\left[\abs{w} \ast \phi_\varepsilon \right] w \right. \\
\nonumber
& & - \left.
\abs{w + v^0} (v^0 + w) + \abs{w} w \right) \D{\tau} \\
\label{eqn:first_summand}
& = &
\int_0^t e^{\iu (t - \tau) \partial_{x}^{2}} \left(
\left[\left(\abs{w + v^0} - \abs{w} \right) \ast \phi_\varepsilon -
\left(\abs{w + v^0} - \abs{w} \right) \right]v^0 \right) \D{\tau} \\
\label{eqn:second_summand}
& & +
\int_0^t e^{\iu (t - \tau) \partial_{x}^{2}} \left(
\left[\left(\abs{w + v^0} - \abs{w} \right) \ast \phi_\varepsilon -
\left(\abs{w + v^0} - \abs{w} \right) \right]w \right.\\
\nonumber
& & +
\left.\left(\abs{w} \ast \phi_\varepsilon - \abs{w} \right)
v^0 \right) \D{\tau}.
\end{eqnarray}
The first summand corresponding to $\abs{v^0}^2$ is treated in the same
way as the $G_1^\varepsilon$-term in the proof of Theorem
\ref{thm:lwp_modnls}, i.e. via a Strichartz estimate and Hölder’s
inequality. We arrive at
\begin{eqnarray*}
& & \norm{\int_0^t e^{\iu (t - \tau) \partial_{x}^{2}} \left(
\left[\left(\abs{w + v^0} - \abs{w} \right) \ast \phi_\varepsilon -
\left(\abs{w + v^0} - \abs{w} \right) \right]v^0 \right) \D{\tau}
}_{L^\infty_t L^2_x} \\
& \leq &
\norm{\left(\abs{w + v^0} - \abs{w} \right) \ast \phi_\varepsilon -
\left(\abs{w + v^0} - \abs{w} \right)
}_{L_t^{\frac{4}{3}}L_x^2} \cdot \norm{v^0}_{L_t^\infty L_x^2}.
\end{eqnarray*}
It suffices to show that the first factor above tends to zero, as
$\varepsilon$ tends to zero. For almost every $t \in [0, T]$ we have that
$\left(\abs{w + v^0} - \abs{w} \right) \in L^2$, which implies, due to
the fact that $(\phi_\varepsilon)_{\varepsilon > 0}$ is an approximation
to the identity, that
\begin{equation*}
\norm{\left(\abs{w + v^0} - \abs{w} \right) \ast \phi_\varepsilon -
\left(\abs{w + v^0} - \abs{w} \right)}_{L_x^2}
\xrightarrow{\varepsilon \to 0+} 0.
\end{equation*}
Furthermore, by Young’s inequality,
\begin{equation*}
\norm{\left(\abs{w + v^0} - \abs{w} \right) \ast \phi_\varepsilon -
\left(\abs{w + v^0} - \abs{w} \right)}_{L_x^2}^{\frac{4}{3}} \lesssim
\norm{v^0}_{L_x^2}^{\frac{4}{3}}
\end{equation*}
for every $\varepsilon > 0$ and almost every $t \in [0, T]$. Also,
\begin{equation*}
\int_0^T \norm{v^0(\cdot, \tau)}_{L_x^2}^{\frac{4}{3}} \D{\tau} =
\norm{v^0}_{L^{\frac{4}{3}}_t L^2_x}^{\frac{4}{3}} \lesssim_T
\norm{v^0}_{L^\infty_t L^2_x}^{\frac{4}{3}}
\end{equation*}
and hence the claim follows by the dominated convergence theorem.

The second summand (Equation \eqref{eqn:second_summand}),
corresponding to $\abs{v^0w}$, is treated like the
$G_2^\varepsilon$-term and we arrive at
\begin{eqnarray*}
& &
\norm{\int_0^t e^{\iu (t - \tau) \partial_{x}^{2}} \left[
\left(\left(\abs{w + v^0} - \abs{w} \right) \ast \phi_\varepsilon -
\left(\abs{w + v^0} - \abs{w} \right) \right) w \right] \D{\tau}
}_{L^\infty_t L^2_x} \\
& & +
\norm{\int_0^t e^{\iu (t - \tau) \partial_{x}^{2}} \left[
\left(\abs{w} \ast \phi_\varepsilon - \abs{w} \right)
v^0 \right] \D{\tau}}_{L^\infty_t L^2_x} \\
& \leq &
\norm{
\left(\abs{w + v^0} - \abs{w} \right) \ast \phi_\varepsilon -
\left(\abs{w + v^0} - \abs{w} \right)}_{L_t^1 L_x^2} \cdot
\norm{w}_{L_t^\infty L_x^\infty} \\
& & +
\norm{v^0}_{L_t^\infty L_x^2}
\norm{\abs{w} \ast \phi_\varepsilon - \abs{w}}_{L_t^1 L_x^\infty}.
\end{eqnarray*}
Observe, that $\abs{w}$ is uniformly continuous in the $x$-variable on
the whole of $\R$. Hence, as for \eqref{eqn:first_summand}, the fact
that $(\phi_\varepsilon)_{\varepsilon > 0}$ is an approximation to the
identity implies the convergence to zero of \eqref{eqn:second_summand}.
\end{proof}

\begin{lem}[Smooth solutions for smooth initial data]
\label{lem:smoothness_modnls}
(Cf. \cite[Proposition 3.11]{tao2006}.)
Let $\varepsilon > 0$, $w \in C([0, T], H^\infty(\T))$ and
$v_0 \in \LSS$ and let $v$ denote the unique solution of
\eqref{eqn:cauchy_modnls_smooth}. Then
$v \in C^1([0, T], H^\infty(\R))$ and for any $s > \iv{2}$ one has
\begin{equation}
\label{eqn:gronwall_modnls}
\norm{v}_{L_t^\infty H_x^s} \leq
C \norm{v_0}_{H^s}
\exp\left(\norm{v}_{L_t^1 L_x^\infty} +
T \norm{w}_{C(H^{s + 1}(\T))} \right)
\end{equation}
for some $C = C(\varepsilon, s) > 0$.
\end{lem}
\begin{proof}
We begin by showing that $v \in C([0, T], H^s(\R))$ for any
$s \in \N$. It suffices to prove that the operator
$\opT^\varepsilon$ from Theorem \ref{thm:lwp_modnls} is a self
mapping in $M(R, T') \subseteq H^s$, for a possibly smaller $T' \leq T$. 
To that end, observe that
\begin{eqnarray*}
\norm{\opT^\varepsilon v}_{H^s} & \leq &
\norm{e^{\iu t \partial_{x}^{2}} v_0}_{H^s} +
\norm{\int_0^t e^{\iu (t - \tau) \partial_{x}^{2}}
G^\varepsilon(w, v) \D{\tau}}_{H^s} \\
& \leq &
\norm{v_0}_{H^s} +
\int_0^t \norm{G^\varepsilon(w, v)}_{H^s} \D{\tau}.
\end{eqnarray*}
The first summand fixes $R \approx \norm{v_0}_{H^s}$. For the integrand
in the second summand we have (the variable $\tau$ is omitted in the
notation)
\begin{eqnarray}
\label{eqn:inhomogeneity}
& & \norm{G^\varepsilon(w, v)}_{H^s} \\
\nonumber
& \leq &
\underbrace{\norm{\left(\left[\abs{w + v} - \abs{w} \right] \ast
\phi_\varepsilon \right) v}_{H^s}}_{\eqqcolon I} +
\underbrace{\norm{\left( \abs{w} \ast \phi_\varepsilon \right) v}_{H^s}
}_{\eqqcolon II} +
\underbrace{\norm{\left(\left[\abs{w + v} - \abs{w} \right] \ast
\phi_\varepsilon \right) w}_{H^s}}_{\eqqcolon III}.
\end{eqnarray}
As $H^s(\R)$ is an algebra with respect to point-wise
multiplication, the first summand is estimated against
\begin{equation*}
\norm{\left(\left[\abs{w + v} - \abs{w} \right] \ast
\phi_\varepsilon \right) v}_{H^s} \lesssim
\norm{\left[\abs{w + v} - \abs{w} \right] \ast \phi_\varepsilon}_{H^s}
\norm{v}_{H^s}.
\end{equation*}
The first product above is further estimated via the definition of the
$H^s$-norm as
\begin{equation}
\label{eqn:smoothened_sobolev}
\norm{\left[\abs{w + v} - \abs{w} \right] \ast \phi_\varepsilon}_{H^s}
\lesssim
\norm{\jb{\cdot}^s \FT \phi_\varepsilon}_{L^\infty}
\norm{\abs{w + v} - \abs{w}}_2.
\end{equation}
Further estimating $\norm{v}_2 \leq \norm{v}_{H^s} \leq R$ and recalling
the integral concludes the discussion of this term. The second summand (II)
is treated via Lemma
\ref{lem:fac}:
\begin{equation*}
\norm{\left( \abs{w} \ast \phi_\varepsilon \right) v}_{H^s} \lesssim_s
\norm{\abs{w} \ast \phi_\varepsilon}_{H^{s + 1}(\T)} \norm{v}_{H^s}.
\end{equation*}
We again estimate $\norm{v}_{H^s} \leq R$ and observe for the other
factor that
\begin{eqnarray*}
\norm{\abs{w} \ast \phi_\varepsilon}_{H^{s + 1}(\T)} & \approx &
\sum_{\abs{\alpha} \leq \lceil s + 1 \rceil}
\norm{\abs{w} \ast \left[D^\alpha \phi_\varepsilon \right]}_{L^2(\T)} \\
& \leq &
\norm{w}_{\infty}
\sum_{\abs{\alpha} \leq \lceil s + 1 \rceil}
\norm{D^\alpha \phi_\varepsilon}_{L^1(\R)} \\
& \lesssim_{\varepsilon, s} &
\norm{w}_{H^{s + 1}(\T)}.
\end{eqnarray*}
The last summand (III) is estimated via
\begin{equation*}
\norm{\left(\left[\abs{w + v} - \abs{w} \right] \ast
\phi_\varepsilon \right) w}_{H^s} \lesssim_{\varepsilon, s}
\norm{v}_{H^s} \norm{w}_{H^{s + 1}(\T)}.
\end{equation*}
The proof of the above requires no new techniques and is omitted. All in
all this shows the local well-posedness of
\eqref{eqn:cauchy_modnls_smooth} in
$C([0, T'], H^s)$, where the guaranteed time of existence is
\begin{equation*}
T' \approx_{\varepsilon, s}
\set{\norm{w}_{H^{s + 1}(\T)}^{-1}, \norm{v_0}_{H^s(\R)}^{-1}}.
\end{equation*}
To prove the estimate \eqref{eqn:gronwall_modnls}, we will employ
Lemma \ref{lem:gronwall_integral} (Gronwall’s inequality). To that end,
let $T'$ be now the maximal time of existence of the solution
$v \in C([0, T'), H^s)$. Observe that
\begin{equation*}
\norm{v(\cdot, t)}_{H^s} =
\norm{(\opT^\varepsilon v)(\cdot, t)}_{H^s} \leq
\norm{v_0}_{H^s} +
\int_0^t \norm{G^\varepsilon(w, v)(\cdot, \tau)}_{H^s} \D{\tau}
\qquad \forall t \in [0, T').
\end{equation*}
The integrand above is estimated as in inequality
\eqref{eqn:inhomogeneity}. The first term (I), however, needs retreatment,
as it is quadratic in $\norm{v}_{H^s}$. The algebra property of
$H^s(\R) \cap L^\infty(\R)$ implies
\begin{equation*}
I \leq
\norm{\left(\left[\abs{w + v} - \abs{w} \right] \ast
\phi_\varepsilon \right)}_{H^s} \norm{v}_\infty +
\norm{\left(\left[\abs{w + v} - \abs{w} \right] \ast
\phi_\varepsilon \right)}_\infty \norm{v}_{H^s}.
\end{equation*}
We estimate the first factor in the first summand by
\eqref{eqn:smoothened_sobolev}. For the first factor of the second
summand we have
\begin{equation*}
\norm{\left(\left[\abs{w + v} - \abs{w} \right] \ast
\phi_\varepsilon \right)}_\infty \leq
\norm{\left[\abs{w + v} - \abs{w} \right]}_\infty
\norm{\phi_\varepsilon}_1 \leq \norm{v}_\infty
\end{equation*}
by Young’s inequality. Reinserting the estimates for the terms $(II)$
and $(III)$ yields
\begin{equation*}
\norm{v(\cdot, t)}_{H^s} \lesssim_{s, \varepsilon}
\norm{v_0}_{H^s} +
\int_0^t \left(\norm{v(\cdot, \tau)}_{\infty} +
\norm{w(\cdot, \tau)}_{H^{s + 1}(\T)} \right)
\norm{v(\cdot, \tau)}_{H^s} \D{\tau}.
\end{equation*}
Gronwall’s inequality now implies
\begin{eqnarray*}
\norm{v(\cdot, t)}_{H^s} & \lesssim_{\varepsilon, s} &
\norm{v_0}_{H^s}
\exp \left(\int_0^t \left(\norm{v(\cdot, \tau)}_{\infty} +
\norm{w(\cdot, \tau)}_{H^{s + 1}(\T)} \right) \D{\tau} \right) \\
& \leq &
\norm{v_0}_{H^s} \exp\left(\norm{v}_{L_t^1 L_x^\infty} +
T' \norm{w}_{C(H^{s + 1}(\T))} \right) \qquad \forall t \in [0, T').
\end{eqnarray*}
Thus we see that a blowup cannot occur for any $T' < T$ and so $T' = T$.

This indeed shows that $v \in C([0, T], H^s)$. As $v_0 \in \LSS$ and
$w \in C([0, T], H^\infty(\T))$ are smooth, a classical result from semi-group theory (see \cite[Theorem 4.2.4]{pazy1992}) implies that
$v \in C^1([0, T], H^s)$. Since $s > \iv{2}$ was arbitrary, the proof is
complete.
\end{proof}

\begin{prop}
\label{prop:norm_estimate}
The unique solution $v$ of \eqref{eqn:cauchy_modnls_smooth} from Theorem
\ref{thm:lwp_modnls} satisfies
\begin{equation}
\label{eqn:norm_estimate}
\norm{v(\cdot, t)}_2 \leq
\norm{v_0}_2 \exp \left[\norm{w}_{L^\infty_t L^\infty_x} t \right]
\qquad \forall t \in [0, T].
\end{equation}
\end{prop}

\begin{proof}
Let $w^{n} \in C([0 ,T], H^\infty(\T))$ be functions with the property
\begin{equation*}
\norm{w^n - w}_{C([0,T],H^1(\T))} \xrightarrow{n \to \infty} 0
\end{equation*}
and let
$v_n \xrightarrow{n \to \infty} v_0$ in the $L^2$-norm where
$v_n \in \LSS$ for all $n \in \N$. Moreover, let
$v^{\varepsilon, n} \in C^1([0, T], L^2))$ be the solution of
\eqref{eqn:cauchy_modnls_smooth} with initial data $v_n$ and
nonlinearity $G^{\varepsilon}(w^n, v^{\varepsilon, n})$ (the smoothness of
$v^{\varepsilon, n}$ follows from Lemma \ref{lem:smoothness_modnls}).
We have
\begin{eqnarray}
\nonumber
\iv{2} \frac{\D}{\D{t}} \norm{v^{\varepsilon, n}(\cdot, t)}_2^2 & = &
\Re \dup{\dot{v}^{\varepsilon, n}(\cdot, t)}
{v^{\varepsilon, n}(\cdot, t)} =
\Re \dup{
\iu \partial_{x}^{2} v^{\varepsilon, n} \pm
\iu G^\varepsilon (w^n, v^{\varepsilon, n})}
{v^{\varepsilon, n}} \\
\nonumber
& = &
\underbrace{-\Re \iu \dup{\nabla v^{\varepsilon, n}}
{\nabla v^{\varepsilon, n}}}_{= 0} \\
\nonumber
& & \pm
\Re \iu \dup{(\abs{v^{\varepsilon, n} + w^n} \ast \phi_\varepsilon)
(v^{\varepsilon, n} + w^n) -
(\abs{w^n} \ast \phi_\varepsilon) w^n}{v^{\varepsilon, n}} \\
\nonumber
& = &
\pm \underbrace{\Re
\iu \dup{(\abs{v^{\varepsilon, n} + w^n} \ast \phi_\varepsilon)
v^{\varepsilon, n}}{v^{\varepsilon, n}}}_{= 0} \\
\label{eqn:modsoln_bound}
& & \pm \Re \iu
\dup{
([\abs{v^{\varepsilon, n} + w^n} - \abs{w^n}]
\ast \phi_\varepsilon) w^n}
{v^{\varepsilon, n}}
\end{eqnarray}
and hence
\begin{eqnarray}
\nonumber
\iv{2} \frac{\D}{\D{t}} \norm{v^{\varepsilon, n}(\cdot, t)}_2^2 & \leq &
\abs{\dup{[\abs{v^{\varepsilon, n} + w^n} - \abs{w^n}]
\ast \phi_\varepsilon) w^n}
{v^{\varepsilon, n}}} \\
\nonumber
& \leq &
\norm{[\abs{v^{\varepsilon, n} + w^n} - \abs{w^n}]
\ast \phi_\varepsilon) w^n}_{L^2_x} \norm{v^{\varepsilon, n}}_{L^2_x} \\
\label{eqn:diff_ineq}
& \leq &
\norm{w^n}_{L^\infty_t L^\infty_x} \norm{v^{\varepsilon, n}}_{L^2_x}^2
\end{eqnarray}
for all $t \in [0, T]$. Above, we obtained the first estimate by the
Cauchy-Schwarz inequa\-lity and the second one by Hölder’s inequality,
Young’s inequality and the size estimate. By the
differential form of the Gronwall’s inequality from Lemma
\ref{lem:gronwall_diff}, we obtain
\begin{equation*}
\norm{v^{\varepsilon, n}(\cdot, t)}_2 \leq
\norm{v_n}_2 \exp \left[\norm{w^n}_{L^\infty_t L^\infty_x} t \right]
\qquad \forall t \in [0, T].
\end{equation*}
In the limit $n \to \infty$, the right-hand side above converges to
the right-hand side of \eqref{eqn:norm_estimate}. It remains to show
\begin{equation}
\label{eqn:convergence_ivs}
\norm{v^{\varepsilon, n} - v^{\varepsilon}}_{L^\infty L^2}
\xrightarrow{n \to \infty} 0,
\end{equation}
because then the left-hand side converges to $\norm{v^\varepsilon}_{L^\infty_t L^2_x}$
in the limit $n \to \infty$. Finally, Lemma
\ref{lem:vanishing_smoothing} yields
\begin{equation*}
\norm{v^\varepsilon}_{L^\infty_t L^2_x} \xrightarrow {\varepsilon \to 0}
\norm{v^0}_{L^\infty_t L^2_x}.
\end{equation*}

To prove \eqref{eqn:convergence_ivs}, observe that the linear evolution
poses no problems and hence it suffices to control the integral term
\begin{equation*}
\norm{\int_0^t e^{\iu (t - \tau) \partial_{x}^{2}} \left[
G^\varepsilon(w, v^{\varepsilon}) -
G^\varepsilon(w^n, v^{\varepsilon, n}) \right]
\D{\tau}}_{L^\infty L^2}.
\end{equation*}
To that end, we will split the difference of the nonlinear terms
according to their effective power up to one exception. We begin by
observing that
\begin{eqnarray*}
& & G^\varepsilon(w, v^{\varepsilon}) -
G^\varepsilon(w^n, v^{\varepsilon, n}) \\
& = &
(\abs{w + v^\varepsilon} \ast \phi_\varepsilon)
v^\varepsilon -
(\abs{w^n + v^{\varepsilon, n}} \ast \phi_\varepsilon)
v^{\varepsilon, n} \\
& & +
([\abs{w + v^\varepsilon} - \abs{w}] \ast \phi_\varepsilon) w -
([\abs{w^n + v^{\varepsilon,n}} - \abs{w^n}] \ast \phi_\varepsilon) w^n
\end{eqnarray*}
and gather the first and the second summand, as well as the third and
the last summand. In the first sum we have
\begin{eqnarray*}
& & (\abs{w + v^{\varepsilon}} \ast \phi_\varepsilon) v^\varepsilon -
(\abs{w^n + v^{\varepsilon, n}} \ast \phi_\varepsilon)
v^{\varepsilon, n} \\
& = &
\underbrace{
(\abs{w + v^\varepsilon} \ast \phi_\varepsilon) v^\varepsilon -
(\abs{w + v^\varepsilon}  \ast \phi_\varepsilon)
v^{\varepsilon, n}}_{\eqqcolon I} \\
& & +
\underbrace{
(\abs{w + v^\varepsilon} \ast \phi_\varepsilon) v^{\varepsilon, n} -
(\abs{w^n + v^{\varepsilon, n}} \ast \phi_\varepsilon)
v^{\varepsilon, n}}_{\eqqcolon II},
\end{eqnarray*}
whereas for the second sum
\begin{eqnarray*}
& & ([\abs{w + v^\varepsilon} - \abs{w}] \ast \phi_\varepsilon)w -
([\abs{w^n + v^{\varepsilon, n}} - \abs{w^n}] \ast \phi_\varepsilon) w^n
\\
& = &
\underbrace{
([\abs{w^n + v^\varepsilon} - \abs{w^n}] \ast \phi_\varepsilon) w^n -
([\abs{w^n + v^{\varepsilon, n}} - \abs{w^n}] \ast \phi_\varepsilon) w^n
}_{\eqqcolon III} \\
& & +
\underbrace{
([\abs{w + v^\varepsilon} - \abs{w}] \ast \phi_\varepsilon) w -
([\abs{w^n + v^\varepsilon} - \abs{w^n}] \ast \phi_\varepsilon)
w^n}_{\eqqcolon IV}
\end{eqnarray*}
holds. We now complete the splitting of
$G^\varepsilon(w, v^{\varepsilon, n}) -
G^\varepsilon(w^n, v^{\varepsilon, n})$ into terms of the same effective
powers. We have
\begin{eqnarray*}
I & = & (\abs{w + v^\varepsilon} \ast \phi_\varepsilon)
(v^\varepsilon - v^{\varepsilon, n}) \\
& = &
([\abs{w + v^\varepsilon} - \abs{w}] \ast \phi_\varepsilon)
(v^\varepsilon - v^{\varepsilon, n}) +
(\abs{w} \ast \phi_\varepsilon)(v^\varepsilon - v^{\varepsilon, n}), \\
II & = &
([\abs{w + v^\varepsilon} -
\abs{w^n + v^{\varepsilon, n}}] \ast \phi_\varepsilon)
v^{\varepsilon, n} \\
& = &
([\abs{w + v^\varepsilon} -
\abs{w + v^{\varepsilon, n}}] \ast \phi_\varepsilon)
v^{\varepsilon,n} +
([\abs{w + v^{\varepsilon, n}} -
\abs{w^n + v^{\varepsilon, n}}]  \ast \phi_\varepsilon)
v^{\varepsilon, n}, \\
III & = &
([\abs{w^n + v^\varepsilon} - \abs{w^n + v^{\varepsilon, n}}] \ast
\phi_\varepsilon) w^n \text{ and} \\
IV & = &
([\abs{w + v^\varepsilon} - \abs{w}] \ast \phi_\varepsilon) w -
([\abs{w + v^\varepsilon} - \abs{w}] \ast \phi_\varepsilon) w^n \\
& & -
([\abs{w^n + v^\varepsilon} - \abs{w^n}] \ast \phi_\varepsilon) w^n +
([\abs{w + v^\varepsilon} - \abs{w}] \ast \phi_\varepsilon) w^n \\
& = &
([\abs{w + v^\varepsilon} - \abs{w}] \ast \phi_\varepsilon)(w - w^n) \\
& & +
([\abs{w + v^\varepsilon} - \abs{w} -
\abs{w^n + v^\varepsilon} + \abs{w^n}] \ast \phi_\varepsilon) w^n,
\end{eqnarray*}
from which the effective powers are obvious, and put
\begin{eqnarray*}
\tilde{G}_1^\varepsilon(w, w^n, v^\varepsilon, v^{\varepsilon, n})
& \coloneqq &
([\abs{w + v^\varepsilon} - \abs{w}] \ast \phi_\varepsilon)
(v^\varepsilon - v^{\varepsilon, n}) \\
& & +
([\abs{w + v^\varepsilon} -
\abs{w + v^{\varepsilon, n}}] \ast \phi_\varepsilon)
v^{\varepsilon, n}, \\
\tilde{G}_2^\varepsilon(w, w^n, v^\varepsilon, v^{\varepsilon, n})
& \coloneqq &
(\abs{w} \ast \phi_\varepsilon)(v^\varepsilon - v^{\varepsilon, n}) +
([\abs{w^n + v^\varepsilon} - \abs{w^n + v^{\varepsilon, n}}] \ast
\phi_\varepsilon) w^n
\\
& &
+ ([\abs{w + v^{\varepsilon, n}} - \abs{w^n + v^{\varepsilon, n}}] \ast
\phi_\varepsilon) v^{\varepsilon, n} \\
& & +
([\abs{w + v^\varepsilon} - \abs{w}] \ast \phi_\varepsilon)(w - w^n) \\
& & +
([\abs{w + v^\varepsilon} - \abs{w} -
\abs{w^n + v^\varepsilon} + \abs{w^n}] \ast \phi_\varepsilon) w^n.
\end{eqnarray*}
Now, by the triangle inequality and the inhomogeneous Strichartz
estimate, one has
\begin{eqnarray*}
& & \norm{\int_0^t e^{\iu (t - \tau) \partial_{x}^{2}} \left[
G^\varepsilon(w, v^{\varepsilon, n}) -
G^\varepsilon(w^n, v^{\varepsilon, n}) \right]
\D{\tau}}_{L^\infty L^2} \\
& \leq &
\norm{\int_0^t e^{\iu (t - \tau) \partial_{x}^{2}}
\tilde{G}_1^\varepsilon(w, w^n, v^\varepsilon, v^{\varepsilon, n})
\D{\tau}}_{L^\infty L^2} \\
& & +
\norm{\int_0^t e^{\iu (t - \tau) \partial_{x}^{2}}
\tilde{G}_2^\varepsilon(w, w^n, v^\varepsilon, v^{\varepsilon, n})
\D{\tau}}_{L^\infty L^2} \\
& \lesssim &
\norm{
\tilde{G}_1^\varepsilon(w, w^n, v^\varepsilon, v^{\varepsilon, n})
}_{L^\frac{4}{3}_t L^1_x} + 
\norm{
\tilde{G}_2^\varepsilon(w, w^n, v^\varepsilon, v^{\varepsilon, n})
}_{L^1_t L^2_x}.
\end{eqnarray*}
We begin by estimating the first summand above. In fact, we have
\begin{eqnarray*}
& & \norm{
([\abs{w + v^\varepsilon} - \abs{w}] \ast \phi_\varepsilon)
(v^\varepsilon - v^{\varepsilon, n})
}_{L^\frac{4}{3}_t L^1_x} \\
& \leq &
\norm{t \mapsto \norm{
[\abs{w + v^\varepsilon} - \abs{w}] \ast \phi_\varepsilon}_{L^2_x}
\norm{v^\varepsilon - v^{\varepsilon, n}}_{L^2_x}
}_{\frac{4}{3}} \\
& \leq &
\norm{t \mapsto \norm{v^\varepsilon}_{L^2_x}
\norm{v^\varepsilon - v^{\varepsilon, n}}_{L^2_x}
}_{\frac{4}{3}} \\
& \leq &
T^{\frac{3}{4}} \norm{v^\varepsilon}_{L^\infty_t L^2_x}
\norm{v^\varepsilon - v^{\varepsilon, n}}_{L^\infty_t L^2_x}.
\end{eqnarray*}
by the Cauchy-Schwarz, Young’s and the inverse triangle inequalities
for the space variable and Hölder’s inequality for the time variable.
Choosing $T$ sufficiently small shows that
\begin{equation*}
\norm{
([\abs{w + v^\varepsilon} - \abs{w}] \ast \phi_\varepsilon)
(v^\varepsilon - v^{\varepsilon, n})
}_{L^\frac{4}{3}_t L^1_x}
\leq
\iv{5} \norm{v^\varepsilon - v^{\varepsilon, n}}_{L^\infty_t L^2_x}.
\end{equation*}
For the second term in the definition of $\tilde{G}_1^\varepsilon$ the
same techniques are applied which yield the bound
\begin{equation*}
\norm{
([\abs{w + v^\varepsilon} -
\abs{w + v^{\varepsilon, n}}] \ast \phi_\varepsilon)
v^{\varepsilon,n}}_{L^\frac{4}{3}_t L^1_x} \leq
T^{\frac{3}{4}} \norm{v^{\varepsilon, n}}_{L^\infty_t L^2_x}
\norm{v^\varepsilon - v^{\varepsilon, n}}_{L^\infty_t L^2_x}.
\end{equation*}
By the proof of Theorem \ref{thm:lwp_modnls}, one has
\begin{equation}
\label{eqn:Linfty_L2_uniform_bound}
\norm{v^{\varepsilon, n}}_{L^\infty_t L^2_x} \lesssim
\norm{v_n}_2 \approx \norm{v_0}_2
\end{equation}
and thus choosing $T$ sufficiently
small again implies
\begin{equation*}
\norm{
([\abs{w + v^\varepsilon} -
\abs{w + v^{\varepsilon, n}}] \ast \phi_\varepsilon)
v^{\varepsilon,n}}_{L^\frac{4}{3}_t L^1_x} \leq
\iv{5} \norm{v^\varepsilon - v^{\varepsilon, n}}_{L^\infty_t L^2_x}.
\end{equation*}
The first term in the definition of $\tilde{G}_2^\varepsilon$ is treated
similarly to the above. The same is true for the second term, where we
additionally observe that
\begin{equation}
\label{eqn:Linfty_H1_uniform_bound}
\sup_{n \in \N} \norm{w^n}_{C([0, T], H^1(\T))} < \infty.
\end{equation}
For the third term, we have
\begin{eqnarray*}
& &
\norm{([\abs{w + v^{\varepsilon, n}} - \abs{w^n + v^{\varepsilon, n}}]
\ast \phi_\varepsilon) v^{\varepsilon, n}}_{L^1_t L^2_x} \\
& \leq &
\norm{[\abs{w + v^{\varepsilon, n}} - \abs{w^n + v^{\varepsilon, n}}] \ast
\phi_\varepsilon}_{L^\infty_t L^\infty_x}
\norm{v^{\varepsilon, n}}_{L^\infty_t L^2_x} \\
& \leq &
\norm{w - w^n}_{L^\infty_t L^\infty_x}
\norm{v^{\varepsilon, n}}_{L^\infty_t L^2_x} \\
& \lesssim &
\norm{w - w^n}_{L^\infty_t H^1_x(\T)} \xrightarrow{n \to \infty} 0,
\end{eqnarray*}
where the Cauchy-Schwarz inequality was used for the first estimate,
the embedding $L^\infty_t \hookrightarrow L^1_t$, Young’s inequality and
the inverse triangle inequality for the second estimate and the
embedding $C([0,T], H^1(\T)) \hookrightarrow L^\infty_t L^\infty_x$
together with \eqref{eqn:Linfty_L2_uniform_bound} for the last
estimate. By the same techniques, one obtains the convergence of the
fourth term to zero.

Finally, for the last term in the definition of
$\tilde{G}_2^\varepsilon$, one has
\begin{eqnarray*}
& & \norm{([\abs{w + v^\varepsilon} - \abs{w} -
\abs{w^n + v^\varepsilon} + \abs{w^n}] \ast \phi_\varepsilon) w^n
}_{L^1_t L^2_x} \\
& \leq &
\norm{\abs{w + v^\varepsilon} - \abs{w} -
\abs{w^n + v^\varepsilon} + \abs{w^n}}_{L^1_t L^2_x}
\norm{w^n}_{L^\infty H^1_x(\T)} \\
& \lesssim &
\norm{\abs{w + v^\varepsilon} - \abs{w} -
\abs{w^n + v^\varepsilon} + \abs{w^n}}_{L^1_t L^2_x},
\end{eqnarray*}
where Hölder’s inequality, the embedding
$C([0,T], H^1(\T)) \hookrightarrow L^\infty_t L^\infty_x$ and Young’s
inequality were used for the first estimate and
\eqref{eqn:Linfty_H1_uniform_bound} for the second estimate. Observe
that by the inverse triangle inequality, the bound
\begin{equation*}
\abs{\abs{w + v^\varepsilon} - \abs{w} -
\abs{w^n + v^\varepsilon} + \abs{w^n}} \leq
2 \min \set{\abs{w - w^n}, \abs{v^\varepsilon}} \leq
2 \abs{v^\varepsilon}
\end{equation*}
holds pointwise (in $t$ and $x$). This implies that
\begin{equation*}
\abs{w + v^\varepsilon} - \abs{w} -
\abs{w^n + v^\varepsilon} + \abs{w^n} \xrightarrow{n \to \infty} 0
\end{equation*}
and hence, by the theorem of dominated convergence for the space
variable,
\begin{equation*}
g_n(t) \coloneqq \norm{\abs{w + v^\varepsilon} - \abs{w} -
\abs{w^n + v^\varepsilon} + \abs{w^n}}_{L^2_x}
\xrightarrow{n \to \infty} 0 \qquad \forall t \in [0, T].
\end{equation*}
Moreover, for all $t \in [0, T]$, we have
$g_n(t) \leq 2 \norm{v^\varepsilon(\cdot, t)}_2$ and
$\norm{v^\varepsilon}_{L^1_t L^2_x} \lesssim
\norm{v^\varepsilon}_{L^\infty_t L^2_x} < \infty$. Hence, reapplying the
theorem of dominated convergence for the time variable yields
\begin{equation*}
\norm{([\abs{w + v^\varepsilon} - \abs{w} -
\abs{w^n + v^\varepsilon} + \abs{w^n}] \ast \phi_\varepsilon) w^n
}_{L^1_t L^2_x} \xrightarrow{n \to \infty} 0
\end{equation*}
as claimed.
\end{proof}

Notice that \eqref{eqn:norm_estimate} together with the local
well-posedness of NLS \eqref{eqn:cauchy_modnls} from Theorem \ref{thm:lwp_modnls} imply that NLS \eqref{eqn:cauchy_modnls}
is globally well-posed, i.e. Theorem \ref{thm:mainthm2} is proved.

\begin{rem}
\label{rem:failure_global}
Observe, that in the case $\alpha \neq 2$, the proof would proceed
roughly unchanged up to Equation \eqref{eqn:modsoln_bound}. However, we could replace the
differential inequality \eqref{eqn:diff_ineq} by
\begin{equation*}
\iv{2} \frac{\D}{\D{t}} \norm{v^{\varepsilon, n}(\cdot, t)}_2^2 \lesssim
\norm{w^n}_{L^\infty_t L^\infty_x}^{\alpha - 1}
\norm{v^{\varepsilon, n}}_{L^2_x}^2 +
\norm{w^n}_{L^\infty_t L^\infty_x} \norm{v^{\varepsilon, n}
}_{L^{\alpha}_x}^\alpha
\end{equation*}
and this bound is not sufficient to exclude a blow-up of the $L^2$-norm.
\end{rem}

\appendix
\section{Quadratic and Subquadratic NLS on the torus}
To prove global existence of solutions to the Cauchy problem of the quadratic and 
subquadratic nonlinear Schrödinger equation on $\T$ (that is
\eqref{eqn:cauchy_nls_torus} with $\alpha \in [1, 2]$), we will employ
the mass and energy conservation laws. The justification of conservation
laws requires solutions which are differentiable in time. Again, time
regularity will be obtained from regularity in space. To that end
we will smoothen out the rough quadratic nonlinearity in such a way
that the solutions of the resulting equation still admit suitable
conservation laws. The regularization is slightly different from the
one used in the proof of Theorem \ref{thm:mainthm2}. Let us mention
that the ideas presented here are borrowed from \cite{ginibre1979}
where the same problem was studied on $\R^d$, using a contraction
argument and conservation laws. Since our setting is based on the
torus, we have to work with Bourgain spaces. For the convenience of
the reader, we present some of the arguments in detail.

Observe that, if $w$ is a sufficiently nice $2 \pi$-periodic function
and $\varepsilon > 0$, then
\begin{eqnarray*}
(w \ast \phi_\varepsilon)(x) & = &
\int_{-\infty}^\infty w(y) \phi_\varepsilon(x - y) \D{y} =
\sum_{n \in \Z}
\int_{(2n - 1) \pi}^{(2n + 1)\pi} w(y)
\phi_\varepsilon(x - y) \D{y} \\
& = &
\int_{-\pi}^{\pi} w(y)
\sum_{n \in \Z} \phi_\varepsilon(x - y - 2n \pi) \D{y}.
\end{eqnarray*}
Hence, convolution of $w$ with $\phi_\varepsilon$ on $\R$
corresponds to convolution of $w$ with the periodization of
$\phi_\varepsilon$ on $\T$. For the rest of the paper we will slightly
abuse the notation and denote this periodization also by
$\phi_\varepsilon$. In the same spirit we will use from now on $\ast$ to
denote the convolution on $\T$.

The smooth version of \eqref{eqn:cauchy_nls_torus} for
$\alpha \in [1, 2]$ reads as
\begin{equation}
\label{eqn:cauchy_nls_smooth_T}
\left\{
\begin{IEEEeqnarraybox}[][c]{rCl}
\iu w_t (x, t) + \partial_x^2 w (x,t) \pm
(\abs{w \ast \phi_\varepsilon}^{\alpha - 1}(w \ast \phi_\varepsilon))
\ast \phi_\varepsilon & = & 0 \qquad
(x, t) \in \T \times \R, \\
w(\cdot, 0) & = & w_0 \ast \phi_\varepsilon
\end{IEEEeqnarraybox}
\right.
\end{equation}
and the corresponding Duhamel’s formula is
(cf. \cite[Equations (2.14), (2.13), (2.11) and (1.15)]{ginibre1979})
\begin{equation}
\label{eqn:duhamel_nls_smooth_T}
w(\cdot, t) =
e^{\iu t \partial_{x}^{2}} (w_0 \ast \phi_\varepsilon) \pm
\iu \int_0^t e^{\iu (t - \tau) \partial_{x}^{2}}
\left[\left(
\abs{w \ast \phi_\varepsilon}^{\alpha - 1}
(w \ast \phi_\varepsilon)\right) \ast \phi_\varepsilon
(\cdot, \tau)\right]
\D{\tau}.
\end{equation}
From now on, we denote by $\FT$ and $\FT^{(-1)}$ the \emph{Fourier
transform} and the inverse Fourier transform, on the torus, respectively. We use the
symmetric choice of constants and write also
\begin{eqnarray*}
\hat{f}(\xi) & \coloneqq & (\FT f)(\xi) =
\iv{\sqrt{2 \pi}}
\int_{-\pi}^\pi e^{-\iu \xi \cdot x} f(x) \D{x}, \\
\check{g}(x) & \coloneqq & \left(\FT^{(-1)} g\right)(x) =
\iv{\sqrt{2 \pi}}
\sum_{\xi \in \Z} e^{\iu \xi \cdot x} g(\xi).
\end{eqnarray*}
One has $\FT (f \ast g) = \sqrt{2 \pi} \hat{f} \hat{g}$. Furthermore,
let $\jb{\xi} \coloneqq \sqrt{1 + \abs{\xi}^2}$ for any $\xi \in \R$ and
$J^s w \coloneqq \FT^{(-1)} \jb{\cdot}^s \FT w$ for any
$w \in (C^\infty(\T))'$.

\subsection{Prerequisites}
In this subsection, we present some technical results from the literature, needed for treatment of the quadratic nonlinearity. 
\begin{lem}
\label{lem:transference}
Let $p \in [1, \infty]$ and $\varepsilon \geq 0$. Then for any
$w \in L^p(\T)$ one has
\begin{equation*}
\norm{w \ast \phi_\varepsilon}_{L^p(\T)} \leq
\norm{w}_{L^p(\T)}.
\end{equation*}
\end{lem}

\begin{lem}
\label{lem:ft_heat_one}
Let $s \in \R$ and $w \in H^s(\T)$. Then
\begin{equation*}
\norm{w \ast \phi_\varepsilon}_{H^s(\T)} \leq
\norm{w}_{H^s(\T)} \qquad \text{and} \qquad
\norm{w \ast \phi_\varepsilon}_{\dot{H}^s(\T)}
\leq \norm{w}_{\dot{H}^s(\T)} \qquad
\forall \varepsilon \geq 0
\end{equation*}
where we denote by $\dot{H}^s(\T)$ the homogeneous Sobolev norm on the torus.
Furthermore, if $\varepsilon > 0$, then
\begin{equation*}
\norm{w \ast \phi_\varepsilon}_{H^s(\T)}  \lesssim_{\varepsilon, s}
\norm{w}_{L^2(\T)}.
\end{equation*}
\end{lem}

\begin{lem}
\label{lem:app_BA}
(Cf. \cite[Theorem 3.16]{brezis2011}.)
Let $w^n \xrightarrow{n \to \infty} w$ in $L^2(\T)$ and assume that 
$\sup_{n \in \N} \norm{w^n}_{H^1(\T)} < \infty$. Then $w \in H^1(\T)$ and
\begin{equation}
\label{eqn:h1_bnd}
\norm{w}_{H^1(\T)} \leq
\liminf_{n \to \infty} \norm{w^n}_{H^1(\T)},
\qquad
\norm{w}_{\dot{H}^1(\T)} \leq
\liminf_{n \to \infty} \norm{w^n}_{\dot{H}^1(\T)},
\end{equation}
and $w^n \rightharpoonup w$ in $H^1(\T)$, i.e. for any $u \in H^1(\T)$
one has
\begin{equation}
\label{eqn:weak_conv}
\lim_{n \to \infty} \dup{w^n}{u}_{H^1(\T)} =
\lim_{n \to \infty}
\sum_{k \in \Z} \jb{k}^2 \overline{\hat{w}^n_k} \hat{u}_k =
\dup{w}{u}_{H^1(\T)}.
\end{equation}
If additionally
$\norm{w^n}_{H^1} \xrightarrow{n \to \infty} \norm{w}_{H^1}$, then
$w^n \xrightarrow{n \to \infty} w$ in $H^1(\T)$.
\end{lem}

In the following we are going to use the $X^{s,b}$ spaces on the torus where $s, b\in\R$. They are defined via the norm (see equation (3.49) in \cite{erdogan2016})
\begin{equation}
\|w\|_{X^{s,b}}=\|\langle k\rangle ^{s}\langle \tau+k^{2}\rangle ^{b}\hat{w}(\tau,k)\|_{L^{2}_{\tau}l^{2}_{k}}.
\end{equation}

\begin{lem}[$X^{0, \frac{3}{8}} \hookrightarrow L^4(\T \times \R)$]
(See \cite[Proposition 2.13]{tao2006}.)
We have
\label{lem:bourgain_lebesgue_embedding}
\begin{equation*}
\norm{w}_{L^4(\T \times \R)} \lesssim
\norm{w}_{X^{0, \frac{3}{8}}}
\end{equation*}
for any $w \in \LSS(\R, C^\infty(\T))$.
\end{lem}

\begin{lem}[$X_\delta^{s, b} \hookrightarrow C(H^s)$]
\label{lem:bs_sobolev_embedding}
(Cf. \cite[Lemma 3.9]{erdogan2016}.)
Let $b > \iv{2}$ and $s \in \R$.
\todo{On what does the implicit constant depend?}{Then}
\begin{equation*}
\norm{w}_{C([0, \delta], H^s(\T))} \lesssim
\norm{w}_{X_\delta^{s, b}}.
\end{equation*}
\end{lem}

\begin{lem}[Linear Schrödinger evolution in $X_\delta^{s, b}$]
\label{lem:bs_linear_estimate}
(Cf. \cite[Lemma 3.10]{erdogan2016}.)
Let $b, s \in \R$, $\delta \in (0, 1]$ and $\eta$ a smooth cut-off in
time. Then
\begin{equation*}
\norm{\eta(t) e^{\iu t \partial_{x}^{2}} w_0}_{X_{\delta}^{s, b}} \lesssim
\norm{w_0}_{H^s(\T)} \qquad \forall w_0 \in H^s(\T).
\end{equation*}
\end{lem}

\begin{lem}[Treating the integral term in $X_\delta^{s, b}$]
\label{lem:bs_integral_estimate}
(Cf. \cite[Lemma 3.12]{erdogan2016}.)
Let $b \in \left(\iv{2}, 1 \right]$, $s \in \R$ and $\delta \leq 1$. Set
$b' \coloneqq b - 1$. Then
\begin{equation*}
\norm{\int_0^t e^{\iu (t - \tau) \partial_{x}^{2}} F(\tau) \D{\tau}
}_{X_\delta^{s, b}} \lesssim_b
\norm{F}_{X_\delta^{s, b'}} \qquad \forall F \in X_\delta^{s, b'}.
\end{equation*}
\end{lem}

\begin{lem}[Changing $b$ in $X_\delta^{s, b}$]
\label{lem:bs_change_b}
(Cf. \cite[Lemma 3.11]{erdogan2016}.)
Let $b, b' \in \left(-\iv{2}, \iv{2} \right)$ with $b' < b$,
$s \in \R$ and $\delta \in (0, 1]$. Then
\begin{equation*}
\norm{w}_{X_\delta^{s, b'}} \lesssim
\delta^{b - b'} \norm{w}_{X_\delta^{s, b}} \qquad \forall w.
\end{equation*}
\end{lem}
The next proposition appears in \cite{erdogan2016} for the case of the
cubic nonlinearity and $\varepsilon = 0$. Since we need the
corresponding result for (sub)quadratic nonlinearities which are more
complicated than the algebraic cubic nonlinearity, we present the proof, too.

\begin{prop}[Control of the nonlinearity in $X_\delta^{s, b}$]
\label{prop:bs_nonlinearity_control}
(Cf. \cite[Proposition 3.26]{erdogan2016}.)
Let $s \geq 0$ and $\varepsilon > 0$ or $\varepsilon = s = 0$. Then,
for all $w_1, w_2$ we have
\begin{eqnarray*}
& &
\norm{
\left(\abs{w_1 \ast \phi_\varepsilon}^{\alpha - 1}
(w_1 \ast \phi_\varepsilon)\right)
\ast \phi_\varepsilon -
\left(\abs{w_2 \ast \phi_\varepsilon}^{\alpha - 1}
(w_2 \ast \phi_\varepsilon)\right)
\ast \phi_\varepsilon
}_{X_\delta^{s, -\frac{3}{8}}} \\
& \lesssim_{\varepsilon, s} &
\left(\norm{w_1}^{\alpha - 1}_{X_\delta^{0, \frac{3}{8}}} +
\norm{w_2}^{\alpha - 1}_{X_\delta^{0, \frac{3}{8}}}\right)
\left(\norm{w_1 - w_2}_{X_\delta^{0, \frac{3}{8}}} \right).
\end{eqnarray*}
\end{prop}
\begin{proof}
Fix $w_1, w_2$. Then, by Plancherel theorem and duality in
$L^2(\R \times \T)$, one has
\begin{eqnarray*}
& & \norm{
\left(\abs{w_1 \ast \phi_\varepsilon}^{\alpha - 1}
(w_1 \ast \phi_\varepsilon)\right)
\ast \phi_\varepsilon -
\left(\abs{w_2 \ast \phi_\varepsilon}^{\alpha - 1}
(w_2 \ast \phi_\varepsilon)\right)
\ast \phi_\varepsilon
}_{X_\delta^{s, -\frac{3}{8}}} \\
& = &
\sup_{\norm{w}_{X_\delta^{-s, \frac{3}{8}}} = 1}
\abs{\dup{
\left(\abs{w_1 \ast \phi_\varepsilon}^{\alpha - 1}
(w_1 \ast \phi_\varepsilon) -
\abs{w_2 \ast \phi_\varepsilon}^{\alpha - 1}
(w_2 \ast \phi_\varepsilon)\right)
\ast \phi_\varepsilon}{w}_{L^2(\R \times \T)}}.
\end{eqnarray*}
Fix any $w \in X_\delta^{-s, \frac{3}{8}}$ with
$\norm{w}_{X_\delta^{-s, \frac{3}{8}}} = 1$. Then
\begin{eqnarray*}
& & \abs{\dup{
\left(\abs{w_1 \ast \phi_\varepsilon}^{\alpha - 1}
(w_1 \ast \phi_\varepsilon) -
\abs{w_2 \ast \phi_\varepsilon}^{\alpha - 1}
(w_2 \ast \phi_\varepsilon)\right)
\ast \phi_\varepsilon
}{w}_{L^2(\R \times \T)}} \\
& = &
\abs{\dup{J^s \left[
\left(\abs{w_1 \ast \phi_\varepsilon}^{\alpha - 1}
(w_1 \ast \phi_\varepsilon) -
\abs{w_2 \ast \phi_\varepsilon}^{\alpha - 1}
(w_2 \ast \phi_\varepsilon)\right)
\ast \phi_\varepsilon
\right]}{J^{-s} w}_{L^2(\R \times \T)}} \\
& \leq &
\norm{
\left(\abs{w_1 \ast \phi_\varepsilon}^{\alpha - 1}
(w_1 \ast \phi_\varepsilon) -
\abs{w_2 \ast \phi_\varepsilon}^{\alpha - 1}
(w_2 \ast \phi_\varepsilon)\right)
\ast (J^s \phi_\varepsilon)
}_{L^\frac{4}{3}(\R \times \T)} \\
& & \cdot
\norm{J^{-s} w}_{L^4(\R \times \T)} \\
& \lesssim_{\varepsilon, s} &
\norm{\abs{w_1 \ast \phi_\varepsilon}^{\alpha - 1}
(w_1 \ast \phi_\varepsilon) -
\abs{w_2 \ast \phi_\varepsilon}^{\alpha - 1}
(w_2 \ast \phi_\varepsilon)
}_{L^\frac{4}{3}(\R \times \T)}
\underbrace{\norm{J^{-s} w}_{X_\delta^{0, \frac{3}{8}}}
}_{= \norm{w}_{X_\delta^{-s, \frac{3}{8}}} = 1} \\
& \leq & \norm{
\abs{w_1 \ast \phi_\varepsilon}^{\alpha - 1}
(w_1 \ast \phi_\varepsilon) -
\abs{w_2 \ast \phi_\varepsilon}^{\alpha - 1}
(w_2 \ast \phi_\varepsilon)
}_{L^\frac{4}{3}(\R \times \T)} \\
& \leq &
\norm{\left(
\abs{w_1 \ast \phi_\varepsilon}^{\alpha - 1} +
\abs{w_1 \ast \phi_\varepsilon}^{\alpha - 1} \right)
((w_1 - w_2) \ast \phi_\varepsilon)
}_{L^\frac{4}{3}(\R \times \T)},
\end{eqnarray*}
where, for the first estimate, we used Hölder’s inequality and
Young’s inequality, Lemma \ref{lem:bourgain_lebesgue_embedding}
for the second and the size estimate \eqref{eqn:size_estimate} for the
last inequality. Applying Hölder’s inequality again yields the upper
bound
\begin{equation*}
\left(\norm{
\abs{w_1 \ast \phi_\varepsilon}^{\alpha - 1}}_{L^4(\R \times \T)} +
\norm{
\abs{w_1 \ast \phi_\varepsilon}^{\alpha - 1}}_{L^4(\R \times \T)}
\right)
\norm{
(w_1 - w_2) \ast \phi_\varepsilon
}_{L^2(\R \times \T)}.
\end{equation*}
For the first factor, we apply Hölder’s and Young’s inequalities as well
as the embedding from Lemma \ref{lem:bourgain_lebesgue_embedding} and
arrive at the upper bound of
\begin{equation*}
\norm{w_1}_{L^4(\R \times \T)}^{\alpha -  1} + 
\norm{w_2}_{L^4(\R \times \T)}^{\alpha -  1}
\lesssim
\norm{w_1}_{X^{0, \frac{3}{8}}}^{\alpha -  1} + 
\norm{w_2}_{X^{0, \frac{3}{8}}}^{\alpha -  1}.
\end{equation*}
For the second factor we use Young’s inequality and the definition of
the norm in $X^{0, \frac{3}{8}}$ to arrive at the final estimate
\begin{equation*}
\Big(\norm{w_1}_{X^{0, \frac{3}{8}}}^{\alpha -  1} + 
\norm{w_2}_{X^{0, \frac{3}{8}}}^{\alpha -  1}\Big)\norm{w_1 - w_2}_{X^{0, \frac{3}{8}}}.
\end{equation*}
\end{proof}

\subsection{Results}
First, we consider local wellposedness:
\begin{thm}
\label{thm:lwp_torus}
(Cf. \cite[Theorem 3.27]{erdogan2016} for the cubic NLS.)
Let $\varepsilon > 0$ and $s \geq 0$ or $\varepsilon = s = 0$. Then the
(smoothened) (sub)quadratic NLS \eqref{eqn:cauchy_nls_smooth_T} is \emph{locally}
well-posed in $H^s(\T)$.
\end{thm}

\begin{proof}
It suffices to show that the right-hand side of
\eqref{eqn:duhamel_nls_smooth_T} defines a contractive self-mapping
$\opT: M(R, \delta) \to M(R, \delta)$ for some $R, \delta > 0$, where
\begin{equation*}
M(R, \delta) \coloneqq
\set{w \in Y \big| \, \norm{w}_Y \leq R}
\end{equation*}
and $Y$ is a suitable subspace of $C([0, \delta], H^s(\T))$.

We consider the case $s \geq 1$ first. Put
$Y = C([0, \delta], H^s(\T))$. Due to $e^{\iu t \partial_{x}^{2}}$ being
an isometry on $H^s(\T)$, for any
$t \in \R$, and Lemma \ref{lem:ft_heat_one} we have
\begin{eqnarray*}
& & \norm{\opT w}_Y \\
& \leq &
\norm{e^{\iu t \partial_{x}^{2}} (w_0 \ast \phi_\varepsilon)}_{H^s(\T)} +
\norm{\int_0^t e^{\iu (t - \tau) \partial_{x}^{2}}
\left[
\left(\abs{w \ast \phi_\varepsilon}^{\alpha - 1} (w \ast \phi_\varepsilon)\right)
\ast \phi_\varepsilon
\right] \D{\tau}}_Y \\
& \leq &
\norm{w_0}_{H^s(\T)} +
\delta 
\norm{(\abs{w \ast \phi_\varepsilon}^{\alpha - 1}
(w \ast \phi_\varepsilon)) \ast \phi_\varepsilon}_Y.
\end{eqnarray*}
This suggests the choice $R \approx \norm{w_0}_{H^s}$. Fix
$\tau \in [0, \delta]$. Then,
due to Lemma \ref{lem:ft_heat_one} and the embedding
$H^s(\T) \hookrightarrow L^\infty(\T)$, we have that
\begin{eqnarray*}
& & \norm{((\abs{w \ast \phi_\varepsilon}^{\alpha - 1}
(w \ast \phi_\varepsilon)) \ast \phi_\varepsilon)(\cdot, \tau)
}_{H^s(\T)} \\
& \lesssim_{\varepsilon, s} &
\norm{(\abs{w \ast \phi_\varepsilon}^{\alpha - 1}
(w \ast \phi_\varepsilon))(\cdot, \tau)}_{L^2(\T)} \\
& \leq &
\norm{(w \ast \phi_\varepsilon)(\cdot, \tau)
}_{L^\infty(\T)}^{\alpha - 1}
\norm{(w \ast \phi_\varepsilon)(\cdot, \tau)}_{L^2(\T)} \\
& \lesssim & R^\alpha.
\end{eqnarray*}
By the above, the condition $\norm{\opT w}_Y \leq R$ is satisfied, if
$\delta \lesssim_{\varepsilon, s} R^{1 - \alpha}$. The contraction property of
$\opT$ is shown in the same way, possibly requiring a smaller implicit
constant in the last inequality.

In the case $s \in [0, 1)$ and $\varepsilon > 0$, consider any
$b \in \left(\iv{2}, \frac{5}{8} \right)$ and put
$Y = X_{\delta}^{s, b}$ (by Lemma \ref{lem:bs_sobolev_embedding} one
indeed has $Y \hookrightarrow C([0, \delta], H^s(\T))$). Then, by the
triangle inequality and Lemmata
\ref{lem:bs_linear_estimate} and \ref{lem:bs_integral_estimate} we have
\begin{eqnarray*}
& & \norm{\opT w}_{X_{\delta}^{s, b}} \\
& \leq &
\norm{e^{\iu t \partial_{x}^{2}} (w_0 \ast \phi_\varepsilon)}_{X_{\delta}^{s, b}} +
\norm{\int_0^t e^{\iu (t - \tau) \partial_x^2} \left[
\left(
\abs{w \ast \phi_\varepsilon}^{\alpha - 1}
(w \ast \phi_\varepsilon) \right)
\ast \phi_\varepsilon \right]
\D{\tau}}_{X_{\delta}^{s, b}} \\
& \lesssim &
\norm{w_0}_{H^s(\T)} +
\norm{\left(\abs{w \ast \phi_\varepsilon}^{\alpha - 1}
(w \ast \phi_\varepsilon)\right)
\ast \phi_\varepsilon}_{X_\delta^{s, b - 1}}.
\end{eqnarray*}
This estimate suggests $R \approx \norm{w_0}_{H^s(\T)}$. For the second
summand, apply Lemma \ref{lem:bs_change_b} and Proposition
\ref{prop:bs_nonlinearity_control} (with $w = 0$) to obtain the upper
bound
\begin{eqnarray}
\nonumber
\norm{
\left(\abs{w \ast \phi_\varepsilon}^{\alpha - 1}
(w \ast \phi_\varepsilon)\right)
\ast \phi_\varepsilon}_{X_\delta^{s, b - 1}}
& \lesssim &
\delta^{1 - b - \frac{3}{8}}
\norm{
\left(\abs{w \ast \phi_\varepsilon}^{\alpha - 1}
(w \ast \phi_\varepsilon)\right)
\ast \phi_\varepsilon}_{X_\delta^{s, -\frac{3}{8}}} \\
\label{eqn:bourgain_lebesgue_control}
& \lesssim &
\delta^{1 - b - \frac{3}{8}}
\norm{w}_{X_\delta^{0, \frac{3}{8}}}^{\alpha - 1}
\norm{w}_{X_\delta^{s, \frac{3}{8}}} \\
\nonumber
& \leq &
\delta^{1 - b - \frac{3}{8}}
\norm{w}_{X_\delta^{s, \frac{3}{8}}}^\alpha \leq
\delta^{1 - b - \frac{3}{8}} R^\alpha.
\end{eqnarray}
As the exponent of $\delta$ is positive, we can choose $\delta$ small
enough to make $\opT$ a self-mapping of $M(R, \delta)$. The fact that
$\opT$ is contractive is proven similarly, possibly requiring a smaller
$\delta$.

The remaining case $\varepsilon = s = 0$ is treated exactly as the last
case.
\end{proof}

In order to prove the conservation laws, we need to be able to
approximate by smooth solutions.
\begin{lem}[Smooth solutions for smooth initial data]
\label{lem:smoothness_smooth_nls_T}
(Cf. \cite[Proposition 3.11]{tao2006}.)
Let $\varepsilon > 0$, and $w_0 \in L^2(\T)$ and let $w$ denote the
unique solution of \eqref{eqn:duhamel_nls_smooth_T}. Then
$w \in C([0, \delta], H^\infty(\R))$ and for any $s > \iv{2}$ one has
\begin{equation}
\label{eqn:gronwall_nls_smooth}
\norm{w}_{L_t^\infty H_x^s} \leq
C \norm{w_0}_{L^2}
e^{Ct \norm{w}^{\alpha - 1}_{L_t^\infty H^1}}
\end{equation}
for some $C = C(\varepsilon, s) > 0$.
\end{lem}
\begin{proof}
As $w$ is the solution to \eqref{eqn:duhamel_nls_smooth_T}, one
immediately has
\begin{eqnarray*}
\norm{w(\cdot, t)}_{H^s} & \leq &
\norm{w_0 \ast \phi_\varepsilon}_{H^s} +
\int_0^t
\norm{(\abs{w \ast \phi_\varepsilon}^{\alpha - 1}
(w \ast \phi_\varepsilon)) \ast \phi_\varepsilon}_{H^s} \D{\tau} \\
& \lesssim_{\varepsilon, s} &
\norm{w_0}_{L^2} +
\int_0^t
\norm{\abs{w \ast \phi_\varepsilon}^{\alpha - 1}
(w \ast \phi_\varepsilon)}_{L^2} \D{\tau} \\
& \leq &
\norm{w_0}_{L^2} +
\norm{w \ast \phi_\varepsilon}_{L^\infty_t L^\infty_x}^{\alpha - 1}
\int_0^t \norm{w \ast \phi_\varepsilon}_{L^2} \D{\tau} \\
& \lesssim &
\norm{w_0}_{L^2} +
\norm{w}_{L^\infty_t H^1}^{\alpha - 1}
\int_0^t \norm{w}_{H^s} \D{\tau}.
\end{eqnarray*}
Now \eqref{eqn:gronwall_nls_smooth} follows from Lemma
\ref{lem:gronwall_integral}.
\end{proof}

\begin{thm}
\label{thm:nls_smooth_global_hs}
Let $\varepsilon > 0$ and $s \in [1, \infty)$. Then the smoothened NLS
\eqref{eqn:cauchy_nls_smooth_T} is \emph{globally} well-posed in
$H^s(\T)$.
\end{thm}
\begin{proof}
Local well-posedness has already been shown in Theorem
\ref{thm:lwp_torus} and it remains to show that the solution $w$ exists
globally. By the blow-up alternative, it suffices to see that
$\norm{w(\cdot, t)}_{H^s(\T)}$ cannot explode. Moreover, by Lemma
\ref{lem:smoothness_smooth_nls_T} it suffices to consider $s = 1$. By
the same lemma, one has that
$w \in C([0, \delta], H^\infty(\T))$ and in particular,
$w \in C^1([0, \delta], H^1(\T))$. Hence, the energy
conservation (cf. \cite[Equations (3.14) and (1.18)]{ginibre1979})
\begin{equation}
\label{eqn:energy}
E_\varepsilon(w(\cdot, t)) \coloneqq
\int_{\T}
\iv{2} \abs{\nabla w(x, t)}^2 \mp
\iv{\alpha + 1}
\abs{(w \ast \phi_\varepsilon) (x, t)}^{\alpha + 1} \D{x} =
E_\varepsilon(w_0 \ast \phi_\varepsilon)
\end{equation}
is applicable to $w$. But
\begin{equation}
\label{eqn:homogeneous_h1_energy}
\norm{w(\cdot, t)}_{H^1(\T)}^2 =
\norm{w_0}_2^2 +
2 E_\varepsilon(w_0 \ast \phi_\varepsilon) \pm
\frac{2}{\alpha + 1} \norm{w(\cdot, t)}_{\alpha + 1}^{\alpha + 1}
\end{equation}
and so $\norm{w(\cdot, t)}_{\dot{H}^1(\T)}$ is controlled by
$E_\varepsilon(w_0 \ast \phi_\varepsilon)$ in the defocusing case. In
the focusing case we can
assume w.l.o.g. that $\norm{w(\cdot, t)}_{\dot{H}^1(\T)}^2$ is an
unbounded function of $t$, (otherwise, there is nothing to show) and say
that $\norm{w(\cdot, t)}_{\dot{H}^1(\T)}^2$ is large. Then, by the
Gagliardo-Nirenberg inequality from \cite[Chapter 8, Eqn. (42)]
{brezis2011}, we have
\begin{equation}
\label{eqn:app_gn}
\norm{w(\cdot, t)}_{\alpha + 1}^{\alpha + 1} \lesssim
\norm{w(\cdot, t)}_2^{\frac{\alpha + 3}{2}}
\norm{w(\cdot, t)}_{H^1(\T)}^{\frac{\alpha - 1}{2}} \leq
\iv{2} \norm{w(\cdot, t)}_{H^1(\T)}^2,
\end{equation}
where above we additionally used the mass conservation
\begin{equation*}
\norm{w(\cdot, t)}_{L^2(\T)} = \norm{w(\cdot, 0)}_{L^2(\T)}.
\end{equation*}
Hence, inserting \eqref{eqn:app_gn} into
\eqref{eqn:homogeneous_h1_energy} and rearranging the inequality shows
that the quantity $\norm{w(\cdot, t)}_{H^1(\T)}^2$ is bounded, in contradiction
to the assumption. This completes the proof.
\end{proof}

\begin{thm}
\label{thm:gwp_nls_T_l2}
(Cf. \cite[Theorem 3.28]{erdogan2016} for the cubic NLS.) The Cauchy
problem for the (sub)quadratic periodic NLS (\eqref{eqn:cauchy_nls_torus}
with $\alpha \in [1, 2]$) is \emph{globally} well-posed in $L^2(\T)$ and
the solution $w$ enjoys mass conservation
$\norm{w(\cdot, t)}_{L^2(\T)} = \norm{w_0}_{L^2(\T)}$.
\end{thm}
\begin{proof}
Local well-posedness has already been shown in Theorem
\ref{thm:lwp_torus}. Let $w$ denote this local solution. By the blow-up
alternative, it suffices to show mass conservation. To that end, let us
denote by $w^\varepsilon$ the global solution of
\eqref{eqn:cauchy_nls_smooth_T} for $\varepsilon > 0$ from Theorem
\ref{thm:nls_smooth_global_hs}. We will show that for any
$b \in \left(\iv{2}, \frac{5}{8}\right)$ one has
$\norm{w^\varepsilon - w}_{X_\delta^{0, b}} \to 0$ as
$\varepsilon \to 0+$. To that end, notice that
\begin{eqnarray}
\label{eqn:l2_conv}
& & \norm{w^\varepsilon - w}_{X_\delta^{0, b}} \\
\nonumber
& \leq &
\norm{e^{\iu t \partial_{x}^{2}} (w_0 \ast \phi_\varepsilon - w_0)
}_{X_\delta^{0, b}} \\
& & +
\norm{\int_0^t e^{\iu (t - \tau) \partial_x^2} \left[
(\abs{w^\varepsilon \ast \phi_\varepsilon}^{\alpha - 1}
(w^\varepsilon \ast \phi_\varepsilon)) \ast \phi_\varepsilon -
\abs{w}^{\alpha - 1}w
\right] \D{\tau}
}_{X_\delta^{0, b}} \\
\nonumber
& \lesssim &
\norm{w_0 \ast \phi_\varepsilon - w_0}_{L^2(\T)} +
\norm{
(\abs{w^\varepsilon \ast \phi_\varepsilon}^{\alpha - 1}
(w^\varepsilon \ast \phi_\varepsilon)) \ast \phi_\varepsilon -
\abs{w}^{\alpha - 1}w
}_{X_\delta^{0, b - 1}} \\
\nonumber
& \lesssim &
\norm{w_0 \ast \phi_\varepsilon - w_0}_{L^2(\T)} +
\delta^{1 - b} \norm{
(\abs{w^\varepsilon \ast \phi_\varepsilon}^{\alpha - 1}
(w^\varepsilon \ast \phi_\varepsilon)) \ast \phi_\varepsilon -
\abs{w}^{\alpha - 1}w
}_{X_\delta^{0, 0}},
\end{eqnarray}
where we used the fact that $w$ and $w^\varepsilon$ solve the
corresponding fixed-point equations and Lemmata
\ref{lem:bs_linear_estimate}, \ref{lem:bs_integral_estimate} and
\ref{lem:bs_change_b}.

For the first summand, observe that
\begin{equation*}
\norm{w_0 \ast \phi_\varepsilon - w_0}_{L^2(\T)} =
\norm{
\left(\jb{k}^s\hat{w}_0(k) (\sqrt{2 \pi}
\hat{\phi}_\varepsilon(k) - 1) \right)_k
}_{l^2(\Z)}
\end{equation*}
and the right-hand side above converges to $0$ as
$\varepsilon \rightarrow 0+$ by the dominated convergence theorem and
the definition of $\phi_\varepsilon$.

For the second summand, note that
$X_\delta^{0, 0} = L^2([0, \delta] \times \T)$ and hence
\begin{eqnarray*}
& & \norm{
(\abs{w^\varepsilon \ast \phi_\varepsilon}^{\alpha - 1}
(w^\varepsilon \ast \phi_\varepsilon)) \ast \phi_\varepsilon -
\abs{w}^{\alpha - 1}w
}_{L^2([0, \delta] \times \T)} \\
& \leq &
\norm{
(\abs{w}^{\alpha - 1}w) \ast \phi_\varepsilon - \abs{w}^{\alpha - 1}w
}_{L^2([0, \delta] \times \T)} \\
& & +
\norm{
(\abs{w^\varepsilon \ast \phi_\varepsilon}^{\alpha - 1}
(w^\varepsilon \ast \phi_\varepsilon) -
\abs{w}^{\alpha - 1}w) \ast \phi_\varepsilon
}_{L^2([0, \delta] \times \T)}.
\end{eqnarray*}
The first summand above goes to zero due to
$(\phi_\varepsilon)_\varepsilon$ being an approximation to the identity
on $L^{2 \alpha}(\T)$. The other summand is further estimated by
\begin{eqnarray}
\nonumber
& & \norm{
(\abs{w^\varepsilon \ast \phi_\varepsilon}^{\alpha - 1}
(w^\varepsilon \ast \phi_\varepsilon) -
\abs{w}^{\alpha - 1}w) \ast \phi_\varepsilon
}_{L^2([0, \delta] \times \T)} \\
\nonumber
& \leq &
\norm{
\abs{w^\varepsilon \ast \phi_\varepsilon}^{\alpha - 1}
(w^\varepsilon \ast \phi_\varepsilon) - \abs{w}^{\alpha - 1}w
}_{L^2([0, \delta] \times \T)} \\
\nonumber
& \lesssim &
\norm{\left(\abs{w^\varepsilon \ast \phi_\varepsilon}^{\alpha - 1} +
\abs{w}^{\alpha - 1} \right) (w^\varepsilon \ast \phi_\varepsilon - w)
}_{L^2([0, \delta] \times \T)} \\
\nonumber
& \leq &
\norm{\left(\abs{w^\varepsilon \ast \phi_\varepsilon}^{\alpha - 1} +
\abs{w}^{\alpha - 1} \right) [(w^\varepsilon - w) \ast \phi_\varepsilon]
}_{L^2([0, \delta] \times \T)} \\
& &
\label{eqn:remaining_term} +
\norm{\left(\abs{w^\varepsilon \ast \phi_\varepsilon}^{\alpha - 1} +
\abs{w}^{\alpha - 1} \right) (w \ast \phi_\varepsilon - w)
}_{L^2([0, \delta] \times \T)}.
\end{eqnarray}
Let us introduce the set $A^\varepsilon = \set{\abs{w^\varepsilon} \ast
\phi_\varepsilon \leq 1}$. Then the first summand above is further
estimated by
\begin{eqnarray*}
& & \norm{\left(\abs{w^\varepsilon \ast \phi_\varepsilon}^{\alpha - 1} +
\abs{w}^{\alpha - 1} \right) [(w^\varepsilon - w) \ast \phi_\varepsilon]
}_{L^2([0, \delta] \times \T)} \\
& \leq &
\norm{
\left(\Ind_{A^\varepsilon}
\abs{w^\varepsilon \ast \phi_\varepsilon}^{\alpha - 1} +
\Ind_{A^0}
\abs{w}^{\alpha - 1} \right) [(w^\varepsilon - w) \ast \phi_\varepsilon]
}_{L^2([0, \delta] \times \T)} \\
& & +
\norm{
\left(\Ind_{(A^\varepsilon)^c}
\abs{w^\varepsilon \ast \phi_\varepsilon}^{\alpha - 1} +
\Ind_{(A^0)^c}
\abs{w}^{\alpha - 1} \right) [(w^\varepsilon - w) \ast \phi_\varepsilon]
}_{L^2([0, \delta] \times \T)} \\
& \leq &
\norm{w^\varepsilon - w}_{L^2([0, \delta] \times \T)} +
\norm{w}_{L^4([0, \delta] \times \T)}
\norm{w^\varepsilon - w}_{L^4([0, \delta] \times \T)} \\
& \lesssim &
\left(1 + \norm{w}_{X_\delta^{0, b}} \right)
\norm{w^\varepsilon - w}_{X_\delta^{0, b}},
\end{eqnarray*}
where we used Hölder’s and Young’s inequalities for the penultimate
estimate and Lemma \ref{lem:bourgain_lebesgue_embedding} for the last
step. Recall that in front of this term is $\delta^{1 - b}$ and,
w.l.o.g., $\delta \ll 1$. Hence we can just move it to the left-hand
side of \eqref{eqn:l2_conv}. The treatment of the last remaining term
\eqref{eqn:remaining_term} does not require any new techniques.

By the above, $\norm{w^\varepsilon - w}_{X_\delta^{0, b}} \to 0$ as
$\varepsilon \to 0+$. Applying Lemma \ref{lem:bs_sobolev_embedding}, we
see that
\begin{eqnarray*}
\norm{w}_{C([0, T], L^2(\T))} & \leq &
\limsup_{\varepsilon \to 0+} \left[
\norm{w^\varepsilon - w}_{X_\delta^{0, b}} + 
\norm{w^\varepsilon}_{C([0, T], L^2(\T))} \right] \\
& \leq &
\limsup_{\varepsilon \to 0+}
\left[
\norm{w_0 \ast \phi_\varepsilon}_{L^2(\T)} \right] =
\norm{w_0}_{L^2(\T)}
\end{eqnarray*}
and hence the solution $w$ indeed enjoys mass conservation. This
finishes the proof.
\end{proof}
In addition to mass conservation, we also have conservation of the energy. 
\begin{thm}
\label{thm:gwp_cauchy_quadratic_nls_torus}
(Cf. \cite[Theorem 3.1]{ginibre1979} and
\cite[Theorem 2.1]{lebowitz1988}.) The Cauchy problem for the
(sub)quadratic periodic NLS (\eqref{eqn:cauchy_nls_torus} with
$\alpha \in [1, 2]$) is \emph{globally} well-posed in $H^1(\T)$ and the
solution $w$ enjoys energy conservation $E(w(\cdot, t)) = E(w_0)$.
\end{thm}
\begin{rem}
In \cite{lebowitz1988} it is claimed that the quadratic NLS is
globally well-posed on the torus. They refer to \cite{ginibre1979},
where it is done on the real line. While our proof of Theorem
\ref{thm:gwp_cauchy_quadratic_nls_torus} borrows some ideas from
\cite{ginibre1979}, we believe that in order to be able to do the
torus case, one needs the result of Bourgain \cite{bourgain1993a}, in
particular, the Bourgain spaces, which appeared 5 years after
\cite{lebowitz1988}.
\end{rem}
\begin{proof}
Let $w_0 \in H^1(\T) \subseteq L^2(\T)$. By Theorem
\ref{thm:gwp_nls_T_l2}, the (sub)quadratic periodic NLS has the unique
global solution $w \in \Cb(\R, L^2(\T))$. It remains to show that
$w \in \Cb(\R, H^1(\T))$. To show that for any $t \in \R$ one has
$w(\cdot, t) \in H^1(\T)$ we first prove that
\begin{equation}
\label{eqn:ts_h1_bnd}
\sup_{\varepsilon > 0} \norm{w^\varepsilon}_{C(\R, H^1(\T))} < \infty.
\end{equation}
By calculations similar to those in the proof of Theorem
\ref{thm:gwp_nls_T_l2}, it suffices to prove the corresponding bound
for the energy $E_\varepsilon(w^\varepsilon(\cdot, t))$.

To that end let $w^\varepsilon$ be the unique global solution of the
modified NLS \eqref{eqn:cauchy_nls_smooth_T} for $\varepsilon > 0$ from
Theorem \ref{thm:nls_smooth_global_hs}. The energy conservation from
Equation \eqref{eqn:energy} implies
\begin{equation*}
E_\varepsilon(w^\varepsilon(\cdot, t)) =
E_\varepsilon(w_0 \ast \phi_\varepsilon) =
\iv{2} \norm{w_0 \ast \phi_\varepsilon}_{\dot{H}^1(\T)}^2 \mp
\iv{\alpha + 1} \norm{w_0 \ast \phi_\varepsilon}_{L^3(\T)}^{\alpha + 1}.
\end{equation*}
Observe that by Lemma \ref{lem:ft_heat_one} the first summand above
satisfies
\begin{equation*}
\norm{w_0 \ast \phi_\varepsilon}_{\dot{H}^1(\T)}^2 \leq
\norm{w_0}_{\dot{H}^1(\T)}^2.
\end{equation*}
If the sign of the second summand is negative (focusing case), there
is nothing left to do. If the sign is positive (defocusing case), one
has
\begin{equation*}
\norm{w_0 \ast \phi_\varepsilon}_{\alpha + 1}^{\alpha + 1} \leq
\norm{w_0}_{\alpha + 1}^{\alpha + 1} \leq 
\norm{w_0}_{L^\infty(\T)}^{\alpha - 1} \norm{w_0}_{L^2(\T)}^2
\leq \norm{w_0}_{H^1(\T)}^{\alpha + 1}
\end{equation*}
by Lemma \ref{lem:transference}. Therefore, the bound
\eqref{eqn:ts_h1_bnd} holds.

Assume for now that $t \in [0, \delta]$, where $\delta$ is the
guaranteed time of existence of $w$ in $L^2(\T)$. From the proof of
Theorem \ref{thm:gwp_nls_T_l2}, one has that
\begin{equation}
\label{eqn:l2_unif_conv}
\lim_{\varepsilon \to 0+}
\norm{w^\varepsilon - w}_{C([0, T], L^2(\T))} = 0.
\end{equation}
Hence, from Equations \eqref{eqn:ts_h1_bnd} and \eqref{eqn:l2_unif_conv}
and Lemma \ref{lem:app_BA} it follows that
\begin{equation*}
\norm{w(\cdot, t)}_{H^1(\T)} \leq
\liminf_{\varepsilon \to 0+} \norm{w^\varepsilon(\cdot, t)}_{H^1(\T)} <
\infty.
\end{equation*}
Observe, that by the above we have
\begin{eqnarray*}
& &
\norm{w^\varepsilon(\cdot, t) \ast \phi_\varepsilon - w
}_{L^{\alpha + 1}(\T)}^{\alpha + 1} \\
& \lesssim &
\norm{(w^\varepsilon(\cdot, t) -
w(\cdot, t)) \ast \phi_\varepsilon}_{L^{\alpha + 1}(\T)}^{\alpha + 1} +
\norm{(w(\cdot, t) \ast \phi_\varepsilon - w(\cdot, t)
}_{L^{\alpha + 1}(\T)}^{\alpha + 1} \\
& \leq &
\left(\norm{w^\varepsilon(\cdot, t)}_{L^\infty}^{\alpha - 1} +
\norm{w(\cdot, t)}_{L^\infty}^{\alpha - 1}\right)
\norm{w^\varepsilon(\cdot, t) - w(\cdot, t)}_{L^2(\T)}^2 \\
& & 
+ \norm{(w(\cdot, t) \ast \phi_\varepsilon - w(\cdot, t)
}_{L^{\alpha + 1}(\T)}^{\alpha + 1}
\xrightarrow{\varepsilon \to 0+} 0
\end{eqnarray*}
and hence
\begin{equation*}
E_0(w(\cdot, t)) \leq
\liminf_{\varepsilon \to 0+} E_\varepsilon(w^\varepsilon(\cdot, t)) \leq
E_0(w_0).
\end{equation*}
Interchanging $0$ and $t$ shows the reverse inequality and proves
the energy conservation $E_0(w_0) = E_0(w(\cdot, t))$.

Reiterating the argument proves that $w \in L^\infty(\R, H^1(\T))$.
It remains to show that $w \in C(\R, H^1(\T))$. To that end, observe
that $t \mapsto w(\cdot, t)$ is weakly continuous in $L^2(\T)$. But,
by the above, $\sup_{t \in \R} \norm{w(\cdot, t)}_{H^1(\T)} < \infty$
and hence $t \mapsto w(\cdot, t)$ is weakly continuous in $H^1(\T)$.
By the observation
\begin{equation*}
\norm{w(\cdot, t) - w(\cdot, s)}_{H^1(\T)}^2 = 
\norm{w(\cdot, t)}_{H^1(\T)}^2 + \norm{w(\cdot, s)}_{H^1(\T)}^2 -
2 \Re \dup{w(\cdot, t)}{w(\cdot, s)}_{H^1(\T)},
\end{equation*}
it is enough to show that $t \mapsto \norm{w(\cdot, t)}_{H^1(\T)}$
is continuous. (See \cite[Proposition 3.32]{brezis2011} for this result
in a more general setting.)

To that end, observe that by the mass and energy conservation we have
\begin{eqnarray*}
\norm{w(\cdot, t)}_{H^1(\T)}^2 & = &
2 E(w(\cdot, t)) \pm \frac{2}{\alpha + 1} \norm{w(\cdot, t)
}_{L^{\alpha + 1}(\T)}^{\alpha + 1} +
\norm{w(\cdot, t)}_{L^2(\T)}^2 \\
& = &
2 E_0(w_0) \pm \frac{2}{\alpha + 1} \norm{w(\cdot, t)
}_{L^{\alpha + 1}(\T)}^{\alpha + 1} +
\norm{w_0}_{L^2(\T)}^2.
\end{eqnarray*}
Moreover, for any $t, s \in \R$ we have
\begin{eqnarray*}
& &
\abs{\norm{w(\cdot, t)}_{L^{\alpha + 1}(\T)}^{\alpha + 1} -
\norm{w(\cdot, s)}_{L^{\alpha + 1}(\T)}^{\alpha + 1}} \\
& \lesssim &
\int_{\T} \abs{w(x, t) - w(x, s)}
\left(\abs{w(x, t)}^\alpha + \abs{w(x, s)}^\alpha \right) \D{x} \\
& \lesssim &
\norm{w}_{L^\infty(\R, H^1(\T))}^\alpha 
\norm{w(\cdot, t) - w(\cdot, s)}_{L^2(\T)}.
\end{eqnarray*}
The fact that $w \in \Cb(\R, L^2(\T))$ concludes the argument.
\end{proof}

\section*{Acknowledgments}
Funded by the Deutsche Forschungsgemeinschaft (DFG, German Research
Foundation) – Project-ID 258734477 – SFB 1173. Dirk Hundertmark
thanks Alfried Krupp von Bohlen und Halbach Foundation for their
financial support.

\printbibliography

\end{document}